\documentclass[11pt]{amsart}
\usepackage[arrow, matrix]{xy}
\usepackage{amsmath}
\usepackage{dsfont}
\usepackage{cancel}
\usepackage{amssymb}
\usepackage{amsthm}
\usepackage[mathscr]{eucal}
\usepackage{color}
\input{epsf}
\theoremstyle{definition}
\newtheorem{definition}{Definition}
\newtheorem{theorem}[definition]{Theorem}
\newtheorem{proposition}[definition]{Proposition}
\newtheorem{lemma}[definition]{Lemma}
\newtheorem{corollary}[definition]{Corollary}
\theoremstyle{remark}
\newtheorem{remark}[definition]{Remark}
\newtheorem{example}[definition]{Example}
\newcounter{enumctr}

\newcommand{\N}{\mathbb{N}}
\newcommand{\oR}{\overline{\mathbb{R}}}
\newcommand{\R}{\mathbb{R}}

\newcommand{\cM}{\mathcal{M}}
\newcommand{\cS}{\mathcal{S}}

\newcommand{\eps}{\varepsilon}
\renewcommand{\phi}{\varphi}

\newcommand{\rT}{\mathrm {T}}
\newcommand{\rmd}{\mathrm{d}}
\newcommand{\fa} {\quad \text{for all }\,}
\newcommand{\faa} {\quad \text{for almost all }\,}

\newcommand{\mand}{\text{ and }}

\newcommand{\qandq}{\quad\text{and}\quad}

\newcommand{\set}[1]{\{#1\}}
\newcommand{\setb}[1]{\big\{#1\big\}}
\newcommand{\setB}[1]{\Big\{#1\Big\}}

\newcommand{\norm}[1]{\|#1\|}

\begin{document}
\title[The Bohl spectrum for nonautonomous differential equations]{The Bohl spectrum for linear nonautonomous differential equations}
\author[Thai~Son~Doan, Kenneth J.~Palmer, and Martin Rasmussen]{Thai~Son~Doan, Kenneth J.~Palmer, and Martin Rasmussen}

\dedicatory{Dedicated to the memory of George R.~Sell}

\address{Thai Son Doan: Department of Probability and Statistics, Institute of Mathematics, Vietnam Academy of Science and Technology, Hanoi, Vietnam, dtson@math.ac.vn}

\address{Kenneth J.~Palmer: Department of Mathematics, National Taiwan University, No.~1, Sec.~4, Roosevelt Road, Taipei 106, Taiwan, palmer@math.ntu.edu.tw}

\address{Martin Rasmussen: Department of Mathematics, Imperial College London, 180 Queen's Gate, London SW7 2AZ, United Kingdom, m.rasmussen@imperial.ac.uk}

\thanks{The first author was supported by a Marie-Curie IEF Fellowship, and the third author was supported by an EPSRC Career Acceleration Fellowship EP/I004165/1 (2010--2015)}

\date{\today}

\begin{abstract}
  We develop the Bohl spectrum for nonautonomous linear differential equation on a half line, which is a spectral concept that lies between the Lyapunov and the Sacker--Sell spectrum. We prove that the Bohl spectrum is given by the union of finitely many intervals, and we show by means of an explicit example that the Bohl spectrum does not coincide with the Sacker--Sell spectrum in general even for bounded systems. We demonstrate for this example that any higher-order nonlinear perturbation is exponentially stable (which is not evident from the Sacker--Sell spectrum), but we show that in general this is not true. We also analyze in detail situations in which the Bohl spectrum is identical to the Sacker--Sell spectrum.
\end{abstract}



\maketitle

\section{Introduction}

The stability theory for linear nonautonomous differential equations has its origin in A.M.~Lyapunov's celebrated PhD Thesis~\cite{Lyapunov_92_1}, where he introduces characteristic numbers, so-called \emph{Lyapunov exponents}, which are given by accumulation points of exponential growth rates of individual solutions. It is well-known that in case of negative Lyapunov exponents, the stability of nonlinearly perturbed systems is not guaranteed without an additional regularity condition.

In the 1970s, R.S.~Sacker and G.R.~Sell developed the Sacker--Sell spectrum theory for nonautonomous differential equations. In contrast to the Lyapunov spectrum, the Sacker--Sell spectrum is not a solution-based spectral theory, but rather is based on the concept of an exponential dichotomy, which concerns uniform growth behavior in subspaces and extends the idea of hyperbolicity to explicitly time-dependent systems. If the Sacker--Sell spectrum lies left of zero, then the uniform exponential stability of nonlinearly perturbed systems is guaranteed.

It was shown in \cite{Malkin_51_1} that the regularity condition on Lyapunov exponents can be more robustly replaced by a nonuniform exponential dichotomy. Here the nonuniformity refers to time, and in contrast to that, so-called Bohl exponents, introduced by P.~Bohl in 1914 \cite{Bohl_14_1}, measure exponential growth along solutions uniformly in time. Bohl exponents have been studied extensively in the literature \cite{Daleckii_74_1}, and current research focuses on applications to differential-algebraic equations and control theory \cite{Berger_12_1,Linh_09_1,Anderson_13_1,Wirth_98_1,Hinrichsen_89_1}, and parabolic partial differential equations \cite{Mierczynski_13_1}.
In this paper, we develop the Bohl spectrum as union of all possible Bohl exponents of a nonautonomous linear differential equation on a half line. We show that the Bohl spectrum lies between the Lyapunov and the Sacker--Sell spectrum and that the Bohl spectrum is given by the union of finitely many (not necessarily closed) intervals. Each Bohl spectral interval is associated with a linear subspace, leading to a filtration of subspaces which is finer than the filtration obtained by the Sacker--Sell spectrum.


We show by means of an explicit example that the Bohl spectrum can be a proper subset of the Sacker--Sell spectrum even if the system is bounded.
We analyze in detail situations in which the Bohl spectrum is identical to the Sacker--Sell spectrum, and in particular,
we obtain this for bounded diagonalizable systems, integrally separated systems, and systems with Sacker--Sell point spectrum.
The fact that the Bohl and Sacker--Sell spectra coincide for diagonalizable systems shows that the Bohl spectrum
mainly gives information about the asymptotic behaviour of individual solutions whereas the Sacker--Sell also
embodies information about the relation between different solutions, in particular, whether or not
the angle between solutions is bounded below by a positive number. An interesting problem in this context
is to give necessary and sufficient conditions that the Bohl and Sacker--Sell spectra coincide.

The example referred to above shows that the Sacker--Sell spectrum can extend past zero even when the Bohl spectrum
is given by a negative number. We demonstrate for this example that any higher-order nonlinear perturbation is exponentially stable,
although this not evident from the Sacker--Sell spectrum. In the last section of this paper, we discuss an example with negative Bohl spectrum such that for a certain nonlinear perturbation, the perturbed system is unstable. This means that it is not possible to prove in general that if the Bohl spectrum lies to the left of zero, then any higher-order nonlinear perturbation is exponentially stable. In a forthcoming paper, we will provide additional conditions on the nonlinearities which give a positive answer to this question, even in situations where the Sacker--Sell spectrum intersects the positive half axis.

This paper is organized as follows. In Section~\ref{sec2}, we provide basic material on the Lyapunov and Sacker--Sell spectrum, and in Section~\ref{sec3}, we introduce the Bohl spectrum. Section~\ref{sec4} is devoted to prove the Spectral Theorem, which says that the Bohl spectrum is given by the union of finitely many intervals. We compare the Bohl spectrum and the Sacker--Sell spectrum in Section~\ref{sec5}, and we discuss nonlinear perturbations to linear systems with negative Bohl spectrum in Section~\ref{sec6}.

\section{Lyapunov and Sacker--Sell spectrum}\label{sec2}

In this section, we review the definition and basic properties of the two main spectral concepts for nonautonomous differential equations: the Lyapunov spectrum and the Sacker--Sell spectrum.

We consider a linear nonautonomous differential equation of the form
\begin{equation}\label{Eq1}
  \dot x=A(t)x\,,
\end{equation}
where $A:\R_{0}^{+}\rightarrow \R^{d\times d}$ is a locally integrable matrix-valued function, i.e.~for any $0\le a < b$,  we have $\int_a^b \|A(t)\|\,\rmd t<\infty$. Let $X:\R_0^+\rightarrow \R^{d\times d}$ be the \emph{fundamental matrix} of \eqref{Eq1}, i.e.~$X(\cdot)\xi$ solves \eqref{Eq1} with the initial value condition $x(0)=\xi$, where $\xi\in\R^d$.

The Lyapunov spectrum describes asymptotic growth of individual solutions of \eqref{Eq1}.

\begin{definition}[Lyapunov spectrum]\label{LyapunovSpectrum}
  The \emph{lower} and \emph{upper characteristic Lyapunov exponents} of a particular non-zero solution $X(\cdot)\xi$ of \eqref{Eq1} are defined by
  \begin{displaymath}
    \chi_-(\xi):=\liminf_{t\to\infty}\frac{1}{t}\ln \|X(t)\xi\|
  \end{displaymath}
  and
  \begin{displaymath}
    \chi_+(\xi):=\limsup_{t\to\infty}\frac{1}{t}\ln \|X(t)\xi\|\,.
  \end{displaymath}
  The \emph{Lyapunov spectrum of \eqref{Eq1}} is then defined as
  \begin{displaymath}
    \Sigma_{\rm Lya}:=\bigcup_{\xi\in\R^d\setminus\{0\}} \set{\chi_+(\xi)}\,.
  \end{displaymath}
\end{definition}

It is well-known \cite{Barreira_02_1,Adrianova_95_1} that there exist $n\in \set{1,\dots,d}$ and $\xi_1,\dots,\xi_n\in \R^d\setminus\{0\}$ such that
\begin{displaymath}
  \Sigma_{\rm Lya}
    =
  \bigcup_{i=1}^n \,\set{\chi_+(\xi_i)} \,.
\end{displaymath}

In contrast to the Lyapunov spectrum, the Sacker--Sell spectrum is based on a hyperbolicity concept for nonautonomous differential equations, given by an exponential dichotomy.

\begin{definition}[Exponential dichotomy]
  The linear differential equation \eqref{Eq1} admits an \emph{exponential dichotomy} with growth rate $\gamma\in\R$ if there exist a projector $P\in\R^{d\times d}$, and constants $K\ge 1$ and $\alpha>0$, such that
  \begin{alignat*}{2}
    \norm{X(t)P X^{-1}(s)} &\le K e^{(\gamma-\alpha)(t-s)} &\fa 0\le s \le t\,,\\
    \norm{X(t)(\mathds{1}-P) X^{-1}(s)} &\le K e^{(\gamma+\alpha)(t-s)} &\fa 0\le t \le s\,,
  \end{alignat*}
  where $\mathds{1}$ denotes the unit matrix. In addition, we say that \eqref{Eq1} admits an \emph{exponential dichotomy} with growth rate $\infty$ if there exists a $\gamma\in \R$ such that \eqref{Eq1} admits an \emph{exponential dichotomy} with growth rate $\gamma$ and projector $P=\mathds{1}$, and \eqref{Eq1} is said to admit an \emph{exponential dichotomy} with growth rate $-\infty$ if there exists a $\gamma\in \R$ such that \eqref{Eq1} admits an \emph{exponential dichotomy} with growth rate $\gamma$ and projector $P=0$, the zero matrix.
\end{definition}

The range of the projector $P$ of an exponential dichotomy is called the \emph{pseudo-stable space}, and the null space of the projector $P$ is called a \emph{pseudo-unstable space}. Note that in contrast to the pseudo-unstable space, the pseudo-stable space is uniquely determined for exponential dichotomies on $\R^+_0$ \cite{Rasmussen_09_1}.

The Sacker--Sell spectrum is then given by set of all growth rates $\gamma$ such that the linear system does not admit an exponential dichotomy with growth rate $\gamma$.

\begin{definition}[Sacker--Sell spectrum]
  The \emph{Sacker--Sell spectrum} of the linear differential equation \eqref{Eq1} is defined by
  \begin{align*}
    \Sigma_{\rm SS} := \set{\gamma\in\oR: & \hbox{ \eqref{Eq1} does not admit an exponential dichotomy}\\
    &\hbox{ with growth rate } \gamma}\,,
  \end{align*}
  where $\oR := \R \cup \set{-\infty,\infty}$.
\end{definition}
The Sacker--Sell spectrum was introduced by R.S.~Sacker and G.R.~Sell in \cite{Sacker_78_2} for skew product flows with compact base. It was generalized to nonautonomous dynamical systems with not necessarily compact base in \cite{Aulbach_01_2,Siegmund_02_4} and for systems defined on a half-line in \cite{Rasmussen_09_1}.

The Spectral Theorem (see \cite{Kloeden_11_2,Rasmussen_09_1} for the half-line case) describes the structure of the dichotomy spectrum.

\begin{theorem}[Sacker--Sell Spectral Theorem]\label{theo3}
  For the linear differential equation \eqref{Eq1}, there exists a $k\in\set{1,\dots, d}$ such that
  \begin{displaymath}
    \Sigma_{\rm SS} = [a_1, b_1]\cup\dots\cup [a_k,b_k]
  \end{displaymath}
  with $-\infty\le a_1\le b_1<a_2\le b_2< \dots<a_k\le b_k\le \infty$. In addition, there exists a corresponding filtration
  \[
  \{0\}=\mathcal W_0 \subsetneq\mathcal W_1\subsetneq\mathcal W_2\subsetneq\dots\subsetneq \mathcal W_k =\R^d\,,
  \]
  which satisfies the dynamical characterization
  \begin{displaymath}
    \mathcal W_i= \setB{\textstyle\xi\in \R^d: \sup_{t\in\R^+_0} \|X(t)\xi\|e^{-\gamma t} < \infty}
  \end{displaymath}
  for all $i\in\set{1,\dots,k-1}$ and $\gamma\in (b_i,a_{i+1})$\,.
\end{theorem}

Note that the linear space $\mathcal W_i$ is the pseudo-stable space of the exponential dichotomy with any growth rate taken from the spectral gap interval $(b_i, a_{i+1})$ for $i\in\set{1,\dots,k-1}$.

The following result on Sacker--Sell spectra of upper triangular systems follows from \cite{Battelli_15_1}. Note that such a statement is only true in the half-line case and does not hold for Sacker--Sell spectra on the entire time axis as demonstrated in \cite{Battelli_15_1}.

\begin{proposition}[Sacker--Sell spectrum of upper triangular systems]\label{prop1}
  Suppose that the linear differential equation \eqref{Eq1} is upper triangular, i.e. $a_{ij}(t)=0$ for all $i>j$ and $t\in\R^+_0$, and assume that the off-diagonal elements $a_{ij}(t)$ for all $i<j$ are bounded in $t\in\R_0^+$. Then the Sacker--Sell spectrum of \eqref{Eq1} coincides with that of its diagonal part $\dot x_i=a_{ii}(t)x_i$, $i\in\set{1,\dots,d}$, for which, the spectrum is the union of the intervals
  $[\alpha_i,\beta_i]$. If also the diagonal elements of the matrix $A(t)$ are bounded, then we have the representation
  \begin{equation}\label{bohlrep}
    \alpha_i=\liminf_{t-s\to\infty}{1\over t-s}\int^t_sa_{ii}(u)\,\rmd u \qandq \beta_i=\limsup_{t-s\to\infty}{1\over t-s}\int^t_sa_{ii}(u)\,\rmd u
  \end{equation}
  for all $i\in\set{1,\dots,d}$.
\end{proposition}

\begin{remark}\label{rem1}
  Note that the representation \eqref{bohlrep} does not hold if the diagonal elements of the matrix $A(t)$ are unbounded. As a counter example consider the one-dimensional system
  \begin{displaymath}
    \dot x = a(t)x\,,
  \end{displaymath}
  where $a:\R_0^+\to \R$ is defined by
  \begin{displaymath}
    a(t) = \left\{
    \begin{array}{ccl}
      n & : & t\in[2n,2n+1], n\in\N_0\,,\\
      -2n-1 & : & t\in[2n+1,2n+2], n\in\N_0\,.
    \end{array}
  \right.
  \end{displaymath}
  It follows that $\int^{n+3}_{n}a(u)\,\rmd u \le 0$ for all $n\in\N$, and it can be proved that
  \begin{displaymath}
    \lim_{t-s\to\infty}{1\over t-s} \int^{t}_{s}a(u)\,\rmd u =-\infty \,.
  \end{displaymath}
  However, the Sacker--Sell spectrum is given by $[-\infty,\infty]$, since $a(t)$ is arbitrarily close to $-\infty$ and $\infty$ on intervals of the length one. This shows that the representation \eqref{bohlrep} does not hold for unbounded coefficient matrices.
\end{remark}

\section{The Bohl spectrum}\label{sec3}

We first define the Bohl spectrum for each solution of \eqref{Eq1}. The Bohl spectrum of \eqref{Eq1} is then the union over the Bohl spectra of the solutions.

\begin{definition}[Bohl spectrum]\label{BohlSpectrum} Consider the linear nonautonomous differential equation \eqref{Eq1} in $\R^d$. The \emph{Bohl spectrum of a particular solution $X(\cdot)\xi$}, $\xi\not=0$, of \eqref{Eq1} is defined as
\begin{align*}
  \Sigma_{\xi}:=\Big\{&\lambda\in\oR: \hbox{there exist sequences $\set{t_n}_{n\in\N}$ and $\set{s_n}_{n\in\N}$  }\\
  & \hbox{ with } t_n-s_n\to\infty \hbox{ such that } \lim_{n\to\infty} \tfrac{1}{t_n-s_n}\ln\tfrac{\|X(t_n)\xi\|}{\|X(s_n)\xi\|}=\lambda
  \Big\}.
\end{align*}
The \emph{Bohl spectrum of \eqref{Eq1}} is defined as
\[
  \Sigma_{\rm Bohl}:=\bigcup_{\xi\in\R^d\setminus\{0\}}\Sigma_{\xi}\,.
\]
\end{definition}

\begin{remark}\label{Remark1}
(i) By Definition~\ref{LyapunovSpectrum}, we have $\chi_{-}(\xi), \chi_{+}(\xi)\in \Sigma_{\xi}$ for any $\xi\in\R^d\setminus\{0\}$, and we see that in contrast to looking at the asymptotic behavior at infinity of a solution by
using the Lyapunov exponent, the Bohl spectrum of this solution provides all possible growth
rates of this solution when the length of observation time tends to infinity and the
initial time is arbitrary.

(ii) The \emph{upper} and \emph {lower Bohl exponent} of a solution $X(\cdot)\xi$ are defined by
\[
\overline\beta(\xi):=
\limsup_{t-s\to\infty}\frac{1}{t-s}\ln\frac{\|X(t)\xi\|}{\|X(s)\xi\|}
,\quad
\underline\beta(\xi):=
\liminf_{t-s\to\infty}\frac{1}{t-s}\ln\frac{\|X(t)\xi\|}{\|X(s)\xi\|},
\]
see \cite[p.~171--172]{Daleckii_74_1} and \cite{Barabanov_01_1}. Thus, $\overline\beta(\xi)$ and $\underline\beta(\xi)$ measure the biggest and smallest growth rate of the solution $X(\cdot)\xi$, when the length of observation time tends to infinity, and we have
\[
  \overline\beta(\xi)=\sup\Sigma_{\xi} \qandq \underline\beta(\xi)=\inf\Sigma_{\xi}\,.
\]
We note that the notion of Bohl exponent used in papers on differential algebraic equations and control theory is different (see the references cited in the Introduction).

(iii) The definition of Bohl spectrum is independent of the norm in $\R^d$.

(iv) A different definition of a Bohl spectrum for discrete systems depending on certain invariant splittings was proposed in \cite[Definition~3.8.1]{Poetzsche_10_2}, and another spectrum between the Lyapunov and Sacker--Sell spectrum based on nonuniform exponential dichotomies was introduced in \cite{Chu_15_1}.
\end{remark}
Note that $\overline\beta(\xi)$ can be $\infty$, and $\underline\beta(\xi)$ can be $-\infty$.
For an arbitrarily chosen $a \in \R$, define
\begin{displaymath}
  [-\infty, a]:=(-\infty,a] \cup \set{-\infty}\,,\quad\quad\quad   [a, \infty]  := [a, \infty) \cup \set{\infty}
\end{displaymath}
and
\begin{displaymath}
[-\infty,-\infty]:= \set{-\infty}, \quad\quad\  [\infty,\infty]:= \set{\infty}, \quad\quad \quad [-\infty,\infty] := \oR\,.
\end{displaymath}

The following proposition describes fundamental properties of the Bohl spectrum of a particular solution.

\begin{proposition}\label{Lemma}
  Consider the linear nonautonomous differential equation \eqref{Eq1} in $\R^d$. For all $\xi\in \R^d\setminus\{0\}$, the following statements hold:
  \begin{itemize}
  \item [(i)]  We have the representation
  \begin{align*}
    \Sigma_{\xi}:=&\Big\{\lambda\in\oR: \hbox{ there exist sequences $\set{t_n}_{n\in\N}$ and $\set{s_n}_{n\in\N}$ with }\\
    & \quad t_n-s_n\to\infty \mand s_n\to\infty \hbox{ such that } \lim_{n\to\infty} \tfrac{1}{t_n-s_n}\ln\tfrac{\|X(t_n)\xi\|}{\|X(s_n)\xi\|}=\lambda
    \Big\}\,,
  \end{align*}
  i.e.~in the definition of Bohl spectrum we can always assume $s_n\to\infty$.
    \item [(ii)]$\Sigma_{\xi}=\Sigma_{\lambda\xi}$ for all $\lambda\in \R\setminus\{0\}$.
  \item [(iii)] $\Sigma_{\xi}=\big[\underline\beta(\xi), \overline\beta(\xi)\big]$.
  \item [(iv)] Suppose that there exists a constant $M>0$ such that
  \begin{equation}\label{Boundedness}
    \|A(t)\|\leq M \quad \faa t\in\R_0^+.
  \end{equation}
  Then $\Sigma_{\xi} \subset [-M,M]$.
  \end{itemize}
\end{proposition}
\begin{proof}
  (i) Let $\lambda\in\Sigma_{\xi}$ be arbitrary. Then there exist two sequences $\set{t_n}_{n\in\N}$ and $\set{s_n}_{n\in\N}$ such that $t_n\ge s_n\ge 0$ and
  \begin{equation}\label{Eq6a}
  \lim_{n\to\infty} t_n-s_n =\infty
  \qandq
  \lim_{n\to\infty}
  \frac{1}
  {t_n-s_n}
  \ln\frac{\|X(t_n)\xi\|}{\|X(s_n)\xi\|}=\lambda\,.
  \end{equation}
  To conclude the proof of this part, we need to construct two sequences $\set{\widetilde t_n}_{n\in\N}$ and $\set{\widetilde s_n}_{n\in\N}$ such that
  \begin{equation}\label{Eq6c}
  \lim_{n\to\infty}\widetilde s_n=\infty\,,\quad \lim_{n\to\infty}\widetilde t_n-\widetilde s_n=\infty\,, \quad \lim_{n\to\infty}
  \frac{1}
  {\widetilde t_n-\widetilde s_n}
  \ln\frac{\|X(\widetilde t_n)\xi\|}{\|X(\widetilde s_n)\xi\|}=\lambda\,.
  \end{equation}
  We now consider two separated cases:

  \emph{Case 1}: The sequence $\set{s_n}_{n\in\N}$ is unbounded. Then there exists a subsequence  $\set{s_{k_n}}_{n\in\N}$  of $\set{s_n}_{n\in\N}$ such that $\lim_{n\to\infty} s_{k_n}=\infty$. Letting $\widetilde s_n:=s_{k_n}$ and $\widetilde t_n:=t_{k_n}$. Then these sequences satisfy \eqref{Eq6c}.

  \emph{Case 2}: The sequence $\set{s_n}_{n\in\N}$ is bounded. Let $\Gamma:=\sup_{n\in\N} s_n$, and let $n\in\N$ be an arbitrary positive integer. Since $\lim_{m\to\infty} t_m-s_m=\infty$ and
  \[
  \sup_{m\in\N}\left|\ln\frac{\|X(s_m+n)\xi\|}{\|X(s_m)\xi\|}\right|
  \leq
  \sup_{t\in [0,\Gamma]}\left|\ln\frac{\|X(t+n)\xi\|}{\|X(t)\xi\|}\right|<\infty\,,
  \]
  it follows that
  \[
  \lim_{m\to\infty}
  \frac{1}{t_m-s_m-n}\left|\ln\frac{\|X(s_m+n)\xi\|}{\|X(s_m)\xi\|} \right|=0.
  \]
  Consequently, there exists $k_n\in\N$ such that
  \begin{equation}\label{Eq6d}
  t_{k_n}-s_{k_n}\geq n^2 \qandq \quad \frac{1}{t_{k_n}-s_{k_n}-n}\left|\ln\frac{\|X(s_{k_n}+n)\xi\|}{\|X(s_{k_n})\xi\|} \right|\leq \frac{1}{n}.
  \end{equation}
  Define two sequences $\set{\widetilde t_n}_{n\in\N}$ and $\set{\widetilde s_n}_{n\in\N}$ by
  \[
  \widetilde t_n=t_{k_n} \qandq \widetilde s_n:=s_{k_n}+n \fa n\in\N\,,
  \]
  where $k_n$ satisfies \eqref{Eq6d}. Obviously, $\lim_{n\to\infty}\widetilde s_n=\infty, \lim_{n\to\infty}\widetilde t_n-\widetilde s_n=\infty$. It remains to compute $\lim_{n\to\infty}\frac{1}{\widetilde t_n-\widetilde s_n}\ln\frac{\|X(\widetilde t_n)\xi\|}{\|X(\widetilde s_n)\xi\|}$. By definition of $\set{\widetilde t_n}_{n\in\N}$ and $\set{\widetilde s_n}_{n\in\N}$, we have
  \begin{align*}
  \frac{1}{\widetilde t_n-\widetilde s_n}\ln\frac{\|X(\widetilde t_n)\xi\|}{\|X(\widetilde s_n)\xi\|}
  &=
  \frac{1}{t_{k_n}-s_{k_n}-n}\ln\frac{\|X(t_{k_n})\xi\|}{\|X(s_{k_n}+n)\xi\|}\\[1ex]
  &=
  \frac{1}{t_{k_n}-s_{k_n}-n}
  \left(\ln\frac{\|X(t_{k_n})\xi\|}{\|X(s_{k_n})\xi\|}+\ln\frac{\|X(s_{k_n})\xi\|}{\|X(s_{k_n}+n)\xi\|}\right).
  \end{align*}
  Using \eqref{Eq6d}, we obtain that
  \begin{equation}\label{Eq6e}
  \left|\frac{1}{\widetilde t_n-\widetilde s_n}\ln\frac{\|X(\widetilde t_n)\xi\|}{\|X(\widetilde s_n)\xi\|}-
  \frac{1}{t_{k_n}-s_{k_n}-n}
  \ln\frac{\|X(t_{k_n})\xi\|}{\|X(s_{k_n})\xi\|} \right|\leq \frac{1}{n}.
  \end{equation}
  On the other hand, from $t_{k_n}-s_{k_n}\geq n^2$, we derive that $\lim_{n\to\infty}\frac{t_{k_n}-s_{k_n}}{t_{k_n}-s_{k_n}-n}=1$ and therefore
  \[
  \lim_{n\to\infty}
  \frac{1}{t_{k_n}-s_{k_n}-n}
  \ln\frac{\|X(t_{k_n})\xi\|}{\|X(s_{k_n})\xi\|}
  =
  \lim_{n\to\infty}
  \frac{1}{t_{k_n}-s_{k_n}}
  \ln\frac{\|X(t_{k_n})\xi\|}{\|X(s_{k_n})\xi\|}
  =\lambda,
  \]
  which together with \eqref{Eq6e} implies that the sequences $\set{\widetilde t_n}_{n\in\N}$ and $\set{\widetilde s_n}_{n\in\N}$ satisfy \eqref{Eq6c} and the proof of this part is complete.

  (ii) This assertion follows directly from Definition~\ref{BohlSpectrum}.

  (iii)   Let $a<b$ be in $\Sigma_{\xi}$, and choose $\lambda\in(a,b)$ arbitrarily. Then there exist sequences $\set{t_n}_{n\in\N}, \set{s_n}_{n\in\N}$, $\set{\tau_n}_{n\in\N}$ and $\set{\sigma_n}_{n\in\N}$ such that $t_n-s_n > n$, $\tau_n-\sigma_n>n$,
  \[
  \lim_{n\to\infty}\frac{1}{t_n-s_n}
  \ln\frac{\|X(t_n)\xi\|}{\|X(s_n)\xi\|}=a
  \qandq
  \lim_{n\to\infty}\frac{1}{\tau_n-\sigma_n}
  \ln\frac{\|X(\tau_n)\xi\|}{\|X(\sigma_n)\xi\|}=b\,.
  \]
  Consequently, there exists $N\in\N$ such that for all $n\geq N$,
  \begin{equation}\label{Eq7}
  \frac{1}{t_n-s_n}
  \ln\frac{\|X(t_n)\xi\|}{\|X(s_n)\xi\|}
  <\lambda<
  \frac{1}{\tau_n-\sigma_n}
  \ln\frac{\|X(\tau_n)\xi\|}{\|X(\sigma_n)\xi\|}\,.
  \end{equation}
  Consider the following continuous function $g:[0,1]\to\R$ defined by
  \[
  g(\theta):=\frac{1}{ \theta (t_n-s_n)+(1-\theta)(\tau_n-\sigma_n)}\ln\frac{\|X(\theta t_n+(1-\theta)\tau_n)\xi\|}{\|X(\theta s_n+(1-\theta)\sigma_n)\xi\|}\,.
  \]
  From \eqref{Eq7}, we have $g(0)>\lambda>g(1)$, and by the Intermediate Value Theorem, there exists $\theta_n\in (0,1)$
  such that $g(\theta_n)=\lambda$. This together with the fact that $\lim_{n\to\infty} \theta_n(t_n-s_n)+(1-\theta_n)(\tau_n-\sigma_n)=\infty$ implies that $\lambda\in\Sigma_{\xi}$ and completes the proof.

  (iv) Let $\xi\in\R^d\setminus\{0\}$ be arbitrary. We have the integral equality
  \[
  X(t)\xi=X(s)\xi+\int_s^{t} A(u)X(u)\xi\,\rmd u \fa t\geq s\,.
  \]
  Thus,
  \[
  \|X(t)\xi\|
  \leq
  \|X(s)\xi\|+M\int_s^t \|X(u)\xi\|\,\rmd u\,.
  \]
  Applying Gronwall's inequality yields that
  \[
  \|X(t)\xi\|\leq e^{M(t-s)}\|X(s)\xi\|\fa t\geq s\geq 0\,,
  \]
  which implies that $\sup\Sigma_{\xi}\leq M$. Similarly, $\inf\Sigma_{\xi}\geq -M$. Hence, by~(iii), $\Sigma_{\xi}=[\inf\Sigma_{\xi},\sup\Sigma_{\xi}]\subset [-M,M]$, which completes the proof.
\end{proof}

\begin{proposition}\label{lemma2}
  Consider a linear nonautonomous differential equation $\dot x=A(t)x$ in $\R^d$, and let $x(t)$, $y(t)$  be solutions such that the angle between them is bounded below by
  a positive number. Then if $\alpha\beta\neq 0$, the solutions $t\mapsto\alpha x(t)+\beta y(t)$ all have the
  same Bohl spectrum.
\end{proposition}

\begin{proof}
  We use the Euclidean norm $\|\cdot\|$ on $\R^d$. Without loss of generality, we may assume that $\alpha=1$. So we consider the solutions
  \[ z(t)=x(t)+\beta y(t).\]
  If we define
  \[ e_1(t)=\frac{x(t)}{\|x(t)\|} \qandq e_2(t)=\frac{y(t)}{\|y(t)\|}\,,\]
  we see that
  \[ z(t)=\|x(t)\|e_1(t)+\beta \|y(t)\|e_2(t)\]
  and
  \[ \|x(t)\|
  = (1-\langle e_1(t),e_2(t)\rangle^2)^{-1}[\langle z(t),e_1(t)\rangle-\langle e_1(t),e_2(t)\rangle\cdot \langle z(t),e_2(t)\rangle]\]
  and
  \[ \beta\|y(t)\|
  = (1-\langle e_1(t),e_2(t)\rangle^2)^{-1}[-\langle e_1(t),e_2(t)\rangle\cdot \langle z(t),e_1(t)\rangle+\langle z(t),e_2(t)\rangle]\,.\]
  By the angle assumption, we have $1-\langle e_1(t),e_2(t)\rangle^2\ge \delta$ for some $\delta>0$. This implies
  \[ \|z(t)\|  \le 2\max\{\|x(t)\|, |\beta| \|y(t)\|\} \le \frac{4}{\delta}\|z(t)\|\,.\]
  Now let $z_1(t)$ correspond to $\beta_1$ and $z_2(t)$ to $\beta_2$. Then we note that
  \[ \frac{\|z_1(t)\|}{\|z_2(t)\|} \le \frac{4}{\delta}\frac{\max\{\|x(t)\|, |\beta_1| \|y(t)\|\}}{\max\{\|x(t)\|, |\beta_2| \|y(t)\|\}}
  \le \frac{4R}{\delta r}\,,\]
  where $R=\max\{|\beta_1|,|\beta_2|\}$ and $r=\min\{|\beta_1|,|\beta_2|\}$. Of course, we can interchange the indices $1$ and $2$ here.
  Then
  \[ \frac{\|z_1(t)\|}{\|z_1(s)\|} = \frac{\|z_1(t)\|}{\|z_2(t)\|} \frac{\|z_2(t)\|}{\|z_2(s)\|} \frac{\|z_2(s)\|}{\|z_1(s)\|}
  \le \frac{16R^2}{\delta^2 r^2}\frac{\|z_2(t)\|}{\|z_2(s)\|}=N\frac{\|z_2(t)\|}{\|z_2(s)\|}\,, \]
  where $N:= \frac{16R^2}{\delta^2 r^2}$.
  It follows that
  \[ \frac{1}{t-s}\ln\frac{\|z_1(t)\|}{\|z_1(s)\|}\le \frac{\ln\,N}{t-s}+ \frac{1}{t-s}\ln\frac{\|z_2(t)\|}{\|z_2(s)\|}\,.\]
  Thus,
  \[ \limsup_{t-s\to\infty}\frac{1}{t-s}\ln\frac{\|z_1(t)\|}{\|z_1(s)\|}\le \limsup_{t-s\to\infty}\frac{1}{t-s}\ln\frac{\|z_2(t)\|}{\|z_2(s)\|}\,.\]
  Switching the indices $1$ and $2$, we get equality. Next from
  \[ \frac{1}{t-s}\ln\frac{\|z_2(t)\|}{\|z_2(s)\|}\ge -\frac{\ln\,N}{t-s}+ \frac{1}{t-s}\ln\frac{\|z_1(t)\|}{\|z_1(s)\|}\,,\]
  we get
  \[ \liminf_{t-s\to\infty}\frac{1}{t-s}\ln\frac{\|z_2(t)\|}{\|z_2(s)\|}\ge \liminf_{t-s\to\infty}\frac{1}{t-s}\ln\frac{\|z_1(t)\|}{\|z_1(s)\|}\]
  and switching the indices $1$ and $2$, we get equality also. The conclusion is that $z_1(t)$ and $z_2(t)$ have the same Bohl spectrum.
\end{proof}

\begin{remark}
  We demonstrate that the common Bohl spectrum of the solution $t\mapsto\alpha x(t)+\beta y(t)$ in Proposition~\ref{lemma2} does not depend just on $\Sigma_x$ and $\Sigma_y$. Consider the diagonal system
  \begin{align*}
    \dot x &=0\,,\\
    \dot y&=a(t)y\,,
  \end{align*}
  where $a(t)=1$ if $T_{2k}\le t\le T_{2k+1}$, and $a(t)=-1$ if $T_{2k+1}\le t\le T_{2k+2}$. Here
  $T_k$ is an increasing sequence with $T_0=0$ and $T_{k+1}-T_k\to\infty$ as $k\to\infty$.
  Then if we take the solutions $x(t)=(1,0)$ and $y(t)=\big(0,\exp(\int^t_0a(u)\,\rmd u)\big)$, it is easy to see
  that $\Sigma_x=\{0\}$ and $\Sigma_y=[-1,1]$. By appropriate choice of the sequence $T_k$, we can arrange that $\int^t_0a(u)\,\rmd u\ge 0$ for $t\ge 0$.
  Then if we use the maximum norm in $\R^2$, we see that $\|x(t)+y(t)\|=|y(t)|$ for all $t\ge 0$.
  This means that for all $t\ge s\ge0$, we have
  \[ \frac{1}{t-s}\ln\frac{\|x(t)+y(t)\|}{\|x(s)+y(s)\|}=\frac{1}{t-s}\ln\frac{\|y(t)\|}{\|y(s)\|}.\]
  It follows that $\Sigma_{x+y}=\Sigma_y$.

  On the other hand, again by appropriate choice of the sequence $T_k$, we can arrange that $\int^t_0a(u)\,\rmd u\le 0$
  for $t\ge T_2$. Then if we use the maximum norm in $\R^2$, we see that $\|x(t)+y(t)\|=|x(t)|=1$ for all $t\ge T_2$.
  So for all $t\ge T_2$ and $s\ge T_2$, we get
  \[ \frac{1}{t-s}\ln\frac{\|x(t)+y(t)\|}{\|x(s)+y(s)\|}=\frac{1}{t-s}\ln\frac{1}{1}=0\,,\]
  which implies that $\Sigma_{x+y}=\Sigma_x$.
\end{remark}

\section{Spectral Theorem}\label{sec4}

We prove in this section that the Bohl spectrum of a locally integrable linear nonautonomous differential equation consists of at most finitely many intervals, the number of which is bounded by the dimension of the system, and we associate a filtration of subspaces to these spectral intervals.

\begin{theorem}[Bohl Spectral Theorem]\label{theo2}
  Consider the linear nonautonomous differential equation \eqref{Eq1} in $\R^d$. Then its Bohl spectrum consists of $k$ (not necessarily closed) disjoint
  intervals, i.e.
  \[
  \Sigma_{\rm Bohl}=I_1\cup\dots\cup I_k,
  \]
  where $1\leq k\leq d$ and $I_1,I_2,\dots,I_k$ are ordered intervals. The interval $I_1$ can be unbounded from below, $I_k$ can be unbounded from above, and $I_2,\dots,I_{k-1}$ are bounded. There exists a corresponding filtration
  \begin{equation}\label{Eq8d}
  \{0\}=\mathcal S_0 \subsetneq\mathcal S_1\subsetneq\mathcal S_2\subsetneq\dots\subsetneq \cS_k=\R^d
  \end{equation}
  satisfying the following dynamical characterization
  \begin{equation}\label{Eq8b}
    \mathcal S_i\setminus\{0\}= \setb{\xi\in \R^d: \Sigma_{\xi}\subset \textstyle\bigcup_{j=1}^i I_j }\fa i\in\set{1,\dots,k}\,.
  \end{equation}
\end{theorem}

\begin{proof}
  Let $\lambda\in\oR\setminus\Sigma_{\rm Bohl}$ be arbitrary. Due to Proposition~\ref{Lemma}~(iii), for any $\xi\in\R^d\setminus\{0\}$, either $\Sigma_{\xi}\subset [-\infty,\lambda)$ or $\Sigma_{\xi}\subset(\lambda,+\infty]$ holds. Define
  \begin{equation}\label{StableSpace}
  \mathcal M_{\lambda}:=
  \setb{
  \xi\in\R^d\setminus\{0\}: \Sigma_{\xi}\subset [-\infty,\lambda)  }\cup \set{0}
  \end{equation}
  and
  \[
  \mathcal N_{\lambda}:=
  \setb{
  \xi\in\R^d\setminus\{0\}: \Sigma_{\xi}\subset (\lambda,\infty]
  }\,.
  \]
  Obviously, $\mathcal M_{\lambda}\cup \mathcal N_{\lambda}=\R^d$, and we show now that $\mathcal M_{\lambda}$ is a linear subspace of $\R^d$. Consider $\xi,\eta\in \mathcal M_{\lambda}$ and $\alpha,\beta\in\R$ with
  $\alpha \xi+\beta \eta\not=0$. Since $\Sigma_{\xi},\Sigma_{\eta}\subset [-\infty,\lambda)$, it follows that
  \[
  \limsup_{t\to\infty}\frac{1}{t}\ln \|X(t)\xi\|<\lambda \qandq \limsup_{t\to\infty}\frac{1}{t}\ln \|X(t)\eta\|<\lambda\,,
  \]
  which implies that there exist $K>0$ and $\mu<\lambda$ such that
  \[
    \max\set{\|X(t)\xi\|, \|X(t)\eta\|} \leq K e^{\mu t}\fa t\geq 0\,.
  \]
  Consequently,
  \[
    \|X(t)(\alpha\xi+\beta\eta)\|
    \leq
    K (|\alpha|+|\beta|)
    e^{\mu t}\fa t\geq 0\,.
  \]
  Hence, there exists a sequence $\set{t_n}_{n\in\N}$ tending to infinity with
  \[
    \lim_{n\to\infty}\frac{1}{t_n}\ln\|X(t_n)(\alpha\xi+\beta\eta)\|\in [-\infty,\lambda)\,.
  \]
  Thus, $\Sigma_{\alpha\xi+\beta\eta}\cap [-\infty,\lambda)\not=\emptyset$, and since $\Sigma_{\alpha\xi+\beta\eta}$ is an interval that does not contain $\lambda$,
  it must be a subset of $[-\infty,\lambda)$, and thus,
  we have $\alpha\xi+\beta\eta\in \mathcal M_{\lambda}$.
  Hence, $\mathcal M_{\lambda}$ is a linear subspace of $\R^d$.

  Let $d_0<d_1<\dots<d_n$ be elements of the set $\{\dim(\mathcal M_{\lambda}): \lambda\in \oR\setminus\Sigma_{\rm Bohl}\}$. Depending on whether $\pm\infty\in \Sigma_{\rm Bohl}$ or not, we have the following estimate on the number $n$:
\begin{itemize}
\item[(a)] If $\pm \infty\in \Sigma_{\rm Bohl}$, then $d_0\geq 1$ and $d_n\leq d-1$ and therefore $n\leq d-2$.
\item[(b)] If $-\infty\in \Sigma_{\rm Bohl}$ and $+\infty\not\in\Sigma_{\rm Bohl}$, then $d_0\geq 1$ and therefore $n\leq d-1$.
\item[(c)] If $-\infty\not\in \Sigma_{\rm Bohl}$ and $+\infty\in\Sigma_{\rm Bohl}$, then $d_n\leq d-1$ and therefore $n\leq d-1$.
\item[(d)] If $-\infty\not\in \Sigma_{\rm Bohl}$ and $+\infty\not\in\Sigma_{\rm Bohl}$, then $n\leq d$.
 \end{itemize}
For $i\in \set{0,\dots,n}$, we define
  \begin{equation}\label{Eq8c}
    J_i:=\setb{\lambda\in\oR\setminus \Sigma_{\rm Bohl}: \dim (\mathcal M_{\lambda})=d_i}.
  \end{equation}
We now show that each set $J_i$ is an interval. Let $i\in\set{0,\dots,n}$ and $a<b$ be two elements of $J_i$. We show now that $[a,b]\subset J_i$. By \eqref{StableSpace}, we have $\mathcal M_{a}\subset \mathcal M_{b}$, and using $\dim(\mathcal M_{a})=\dim(\mathcal M_{b})$, this implies that $\mathcal M_{a}=\mathcal M_{b}$, and thus $\mathcal N_{a}=\mathcal N_{b}$. This means that $\R^d=\mathcal M_{a}\cup \mathcal N_b$. Let $\lambda\in [a,b]$ be arbitrary and $\xi\in\R^d\setminus\{0\}$. Thus, either $\xi\in \mathcal M_{a}$ or $\xi\in \mathcal N_b$. In both of these cases, we have $\lambda\not\in\Sigma_{\xi}$ and therefore $\lambda\in \oR\setminus \Sigma_{\rm Bohl}$. Now, we know that $\mathcal M_{\lambda}$ is a linear subspace and by \eqref{StableSpace} we have $\mathcal M_{a}\subset \mathcal M_{\lambda}\subset\mathcal M_{b}$. Thus, $\mathcal M_{a}=\mathcal M_{\lambda}=\mathcal M_{b}$ and therefore $\lambda\in J_i$. This means that we have proved that $J_i$ is an interval. Obviously, the order of the intervals $J_0,\dots,J_n$ is $J_0<J_1<\dots<J_n$ and we have
\[
\Sigma_{\rm Bohl}=\oR\setminus\bigcup_{i=0}^n J_i.
\]
Let $k$ denote the number of disjoint intervals $I_i$ of $\Sigma_{\rm Bohl}$. According to the cases (a)--(d) above, we have the following dependence of $k$ and $n$:
\begin{itemize}
\item[(i)] $k=n+2$ in case (a) above,
\item[(ii)] $k=n+1$ in case (b) and (c) above,
\item[(iii)] $k=n$ in case (d) above.
\end{itemize}
Thus, from the relation between $n$ and $d$ established above, we always obtain that $k\leq d$. To conclude the proof, for each $i\in\{1,\dots,k\}$, we define the set $\mathcal S_i$ as in \eqref{Eq8b} together with $\set{0}$. Note that the space $\cS_i$ coincides with $\cM_\lambda$ for $\lambda = \frac{1}{2} (\sup I_i + \inf I_{i+1})$, where $i\in\set{1,\dots,k-1}$, and $\cS_k=\cM_\lambda=\R^d$ for $\lambda> \sup I_k$. Then, clearly $\mathcal S_i$ is a linear subspace and satisfies \eqref{Eq8d}. This finishes the proof.
\end{proof}
Next, we concentrate on constructing an example of a nonautonomous differential equation such that its Bohl spectrum is not closed. Our construction is implicit by using a result from \cite{Barabanov_94_1}:

Let $\mathcal M_d$ denote the set of all piecewise continuous and uniformly bounded matrix-valued functions $A:\R_0^+\rightarrow \R^{d\times d}$. For each $A\in\mathcal M_d$, consider the linear nonautonomous differential equation
\begin{equation}\label{Eq8}
\dot x=A(t)x\,.
\end{equation}
Consider the \emph{uniform upper exponent function} of \eqref{Eq8}, $\overline\beta_A:\R^{d}\setminus \set{0}\rightarrow \R$, where $\overline\beta_A(\xi)$ is the upper Bohl exponent of the solution $X(t)\xi$ of \eqref{Eq8}. A complete description of the set of functions $\overline{\mathcal B}_d:=\{\overline\beta_A: A\in \mathcal M_d\}$ is given as follows (see \cite[Theorem 1]{Barabanov_94_1}).

\begin{theorem}\label{Theorem_Barabanov}
 A function $\beta:\R^d\setminus\{0\}\rightarrow \R$ belongs to the class $\overline{\mathcal B}_d$ if and only if it satisfies the following three conditions:
\begin{itemize}
\item [(i)] $\beta$ is bounded.
\item [(ii)] $\beta(\xi)=\beta(\alpha\xi)$ for any nonzero $\alpha\in\R$ and $\xi\in\R^d\setminus\{0\}$.
\item [(iii)] For any $q\in\R$, the set $\{\xi: \beta(\xi)\geq q\}$ is a $G_{\delta}$ set.
\end{itemize}
\end{theorem}

The following example shows that the intervals of the Bohl spectrum do not need to be closed.

\begin{example}\label{exam1}
  Consider the function $\beta:\R^2\setminus\{0\}\rightarrow \R$ defined by
  \begin{equation}\label{Eq9}
  \beta(r\cos\phi,r\sin\phi)=
  \left\{
    \begin{array}{ll}
      0 & \hbox{ : } \phi=0, \\
      |\cos\phi |& \hbox{ : } \phi\in (0,2\pi),
    \end{array}
  \right.
  \end{equation}
  where $r\in(0,\infty)$. Obviously, the function $\beta$ satisfies the three conditions of
  Theorem \ref{Theorem_Barabanov}. Consequently, there exists a piecewise continuous and uniformly
  bounded matrix-valued function $A:\R_0^+\rightarrow \R^{2\times 2}$
  such that $\overline{\beta}_A\equiv \beta$. By construction of $\beta_A$,
  it is easy to see that $[0,1)\subset \Sigma_{\rm Bohl}$. Suppose to the contrary that
  $\Sigma_{\rm Bohl}$ is closed. Thus, $1\in \Sigma_{\rm Bohl}$, which means there exists
  $\xi\in\R^2\setminus\{0\}$ such that $1\in \Sigma_{\xi}$. That leads to a contradiction,
  since $\overline\beta_A(\xi)<1$. Thus, $\Sigma_{\rm Bohl}$ is not closed.
\end{example}

In the remaining part of this section, we show that Bohl spectrum is preserved under a kinematic similarity transformation. Recall that a linear nonautonomous differential equation
\begin{equation}\label{Eq10}
  \dot x=A(t)x \fa t\in\R_0^+
\end{equation}
is said to be \emph{kinematically similar} to another linear nonautonomous differential equation
\begin{equation}\label{Eq11}
  \dot y=B(t)y \fa t\in\R_0^+
\end{equation}
if there exists a continuously differentiable function $S:\R^+_0\to\R^{d\times d}$ of invertible matrices such that both $S$ and $S^{-1}$ are bounded, and which satisfies the differential equation
\begin{equation}\label{Eq12}
\dot S(t)=A(t)S(t)-S(t)B(t) \fa t\in \R_0^+
\end{equation}
(see \cite[p.~38]{Coppel_78_1}).

\begin{proposition}[Invariance of the Bohl spectrum under kinematic similarity transformations]
  Suppose that \eqref{Eq10} and \eqref{Eq11} are kinematically similar. Then the Bohl spectra $\Sigma_{\rm Bohl}(A)$ and $\Sigma_{\rm Bohl}(B)$ of \eqref{Eq10} and \eqref{Eq11} coincide.
\end{proposition}
\begin{proof}
Let $X_A(t)$ and $X_B(t)$ denote the fundamental matrix solution of \eqref{Eq10} and \eqref{Eq11}, respectively. From \eqref{Eq12}, we derive
\[
X_B(t)=S(t)^{-1}X_A(t)S(0) \fa t\ge0\,,
\]
which implies that the Bohl spectrum of the solution $X_A(t)S(0)\xi$ of
\eqref{Eq10} is equal to
the Bohl spectrum of the solution $X_B(t)\xi$ for all $\xi\in\R^d\setminus\{0\}$,
where we use the inequality $\ln\|y \|-\ln\|S^{-1}(t)\|\le \ln\|S(t)y\|\le  \ln\|y\|+\ln\|S(t)\|$.
Since $S(0)$ is invertible it follows that $\Sigma_{\rm Bohl}(A)=\Sigma_{\rm Bohl}(B)$
and the proof is complete.
\end{proof}
\section{Bohl and Sacker--Sell spectrum}\label{sec5}

This section is devoted to the comparison of the Bohl spectrum with the Sacker--Sell spectrum.
Note that the one-dimensional example discussed in Remark~\ref{rem1} is an unbounded system for which the both spectra do not coincide, since the Bohl spectrum of this differential equation is given by $\set{-\infty}$.
It follows also directly from Example~\ref{exam1} that the Bohl spectrum does not always coincide with the Sacker--Sell spectrum, since the Sacker--Sell spectrum consists of closed intervals.
In this section, we give an explicit example of a bounded two-dimensional system for which the Sacker--Sell spectrum is a nontrivial interval
and the Bohl spectrum is a single point. We also show that the Bohl spectrum is always
a subset of the Sacker--Sell spectrum, and we provide sufficient conditions under which both spectra coincide.

\subsection{The Bohl spectrum can consist of one point, when the Sacker--Sell spectrum is a non-trivial interval}\label{subsec1}

Consider a $\delta > 0$ and an increasing sequence of non-negative numbers $\set{T_k}_{k\in\N_0}$ satisfying $T_0=0$ and the conditions
\begin{equation}\label{Example_01}
\lim_{k\to\infty} (T_{k+1}-T_k)=\infty\qandq
\lim_{k\to\infty} \frac{e^{T_{2k+2}-T_{2k+1}}}{T_{2k+1}-T_{2k}}=0\,.
\end{equation}
An example of such a sequence $\set{T_k}_{k\in\N}$ is $T_0=0$ and
\[
  T_{k+1}
  :=
  \left\{
  \begin{array}{ccl}
  T_k+ e^{k^2}& : & k \hbox{ is even}\,, \\
  T_k+ k & : & \hbox{otherwise}\,.
  \end{array}
  \right.
\]
Define a piecewise constant matrix-valued function $A:\R_0^+\rightarrow \R^{2\times 2}$ by
\begin{equation}\label{Example_02}
  A(t):=\left\{
  \begin{array}{ccl}
    A_1 &  : & T_{2k}\le t\le T_{2k+1}\,, \\
    A_2 & : &  T_{2k+1}\le t\le T_{2k+2}\,,
  \end{array}\right.
\end{equation}
where
\[
  A_1:=\left(\begin{matrix} -1& \delta\\ 0&-1\end{matrix}\right)\qandq
  A_2:=\left(\begin{matrix} -1& 0\\ 0&0\end{matrix}\right)\,.
\]
\begin{proposition}\label{Non-coincidence}
  Consider the bounded system
  \begin{equation}\label{Eq2}
  \dot x=A(t)x,
  \end{equation}
  where $A:\R_0^+\to\R^{2\times 2}$ is defined as in \eqref{Example_02}. Then the Bohl spectrum $\Sigma_{\rm Bohl}$ and the Sacker--Sell spectrum $\Sigma_{\rm SS}$ of this system are given by
  \[
  \Sigma_{\rm Bohl}=\{-1\} \qandq\Sigma_{\rm SS}=[-1,0]\,,
  \]
  respectively.
\end{proposition}
Before proving the above proposition, we need the following lemma.
\begin{lemma}\label{Lemma1}
  Let $t\mapsto(x(t),y(t))$ be an arbitrary nonzero solution of \eqref{Eq2} with $y(0)\not=0$. Then
  there exists $T>0$ such that $x(t)$ and $y(t)$ have the same sign for all $t\ge T$\,.
\end{lemma}

\begin{proof} The flows for the autonomous systems $\dot x=A_1x$ and $\dot x=A_2x$ are given by
\[
e^{A_1 t}=e^{-t}\left(\begin{matrix} 1& \delta t\\ 0&1\end{matrix}\right) \qandq
e^{A_2 t}=\left(\begin{matrix} e^{-t}& 0\\ 0&1\end{matrix}\right)\,,
\]
respectively. First suppose that $y(0)>0$, and without loss of generality assume that $y(0)=1$. We show by induction that
\begin{equation}\label{rel3}
  x(T_{2k+1}) \ge e^{-T_{2k+1}}\left(x(0)+\delta\sum_{\ell=0}^k (T_{2\ell +1}- T_{2\ell})\right) \fa k\in\N_0\,.
\end{equation}
This is clearly true for $k=0$, since we have $x(T_1)= e^{-T_1}\big(x(0)+\delta T_1\big)$. We now assume that \eqref{rel3} is true for a fixed $k\in\N_0$, and we prove \eqref{rel3} for $k+1$. This follows from
\begin{align*}
  &\quad\,\, x(T_{2k+3}) \\
  &= e^{-(T_{2k+3}-T_{2k+2})}\big(x(T_{2k+2})+\delta(T_{2k+3}-T_{2k+2})y(T_{2k+2})\big)\\
  & = e^{-(T_{2k+3}-T_{2k+1})}x(T_{2k+1}) +e^{-(T_{2k+3}-T_{2k+2})}\delta(T_{2k+3}-T_{2k+2})y(T_{2k+2})\\
  &\ge e^{-T_{2k+3}}\Big(\big(x(0)+\delta\textstyle \sum_{\ell=0}^k (T_{2\ell +1}- T_{2\ell})\big) \\
  & \quad +e^{T_{2k+2}}\delta(T_{2k+3}-T_{2k+2})y(T_{2k+2})\Big)\\
  & \ge e^{-T_{2k+3}}\big(x(0)+\delta\textstyle \sum_{\ell=0}^{k+1} (T_{2\ell +1}- T_{2\ell})\big)\,,
\end{align*}
where the last inequality follows from $e^{-t} y(t) \ge 1$ for all $t\ge 0$.

It follows from \eqref{rel3} that there exists a $k\in\N$ with $x(T_{2k+1})\ge 0$, and we also have $y(T_{2k+1})>0$.
Since the first quadrant is positively invariant under both systems,
it follows that $x(t)\ge 0$ and $y(t)\ge 0$ for $t\ge T_{2k+1}$.
If $y(0)<0$, we get the required result by considering the solution $-(x(t),y(t))$.
\end{proof}
\begin{proof}[Proof of Proposition \ref{Non-coincidence}]
  We first show that $\Sigma_{\rm SS}=[0,1]$. Due to Proposition~\ref{prop1}, the Sacker--Sell spectrum of this upper triangular system coincides with that for the diagonal system diag$(a(t),b(t))$, where $a(t) = -1$ and $-1 \le b(t) \le 0$. $\dot x=a(t)x$ has Sacker--Sell spectrum $\set{-1}$, and the Sacker--Sell spectrum of $\dot y=b(t)y$ is contained in $[-1,0]$. However,
  \[\frac{1}{T_{2k}-T_{2k-1}}\int^{T_{2k}}_{T_{2k-1}}b(t)\,\rmd t
  \to 0\quad{\rm and}\quad
  \frac{1}{T_{2k+1}-T_{2k}}\int^{T_{2k+1}}_{T_{2k}}b(t)\,\rmd t
  \to -1\,,\]
  so the Sacker--Sell spectrum of $\dot y=b(t)y$ is exactly $[-1,0]$, and thus, the Sacker--Sell spectrum of the whole system is $[-1,0]$.

  We now show that $\Sigma_{\rm Bohl}=\{-1\}$ by proving that for fixed
  $\xi\in\R^2\setminus\{0\}$, we have $\Sigma_{\xi}=\{-1\}$.
  For this purpose, write $(x(t),y(t))=X(t)\xi$ and let $\eps\in (0,\delta)$ be arbitrary.
  By \eqref{Example_01} and Lemma \ref{Lemma1}, assuming without loss of generality that $y(0)\neq 0$, there exists a $K\in\N$ such that $x(t)$
  and $y(t)$ have the same sign for $t\ge T_{2K}$ and
    \begin{equation}\label{rel2}
      \frac{e^{T_{2k+2}-T_{2k+1}}}{T_{2k+1}-T_{2k}}\le\eps \fa k\geq K\,.
    \end{equation}
    We show that
    \begin{equation}\label{Eq3}
      \frac{|x(t)|}{|y(t)|}
      \geq
      \frac{\delta}{\eps} >1 \fa t\geq T_{2K+2}\,.
    \end{equation}
  Firstly, since
  \[ X(T_{2k+1})X^{-1}(T_{2k})
  =e^{-(T_{2k+1}-T_{2k})}
  \left(\begin{matrix}
  1 & \delta(T_{2k+1}-T_{2k})\\
  0 & 1
  \end{matrix}\right) \fa k\ge K\,,\]
  it follows that
  \begin{equation}\label{rel1}
  \frac{|x(T_{2k+1})|}{|y(T_{2k+1})|}
  =\frac{|x(T_{2k})+\delta(T_{2k+1}-T_{2k})y(T_{2k})|}{|y(T_{2k})|}
  \ge \delta(T_{2k+1}-T_{2k})\,.
  \end{equation}
  Then for $t\in [T_{2k+2},T_{2k+3}]$ and $k\ge K$, we have
  \begin{align*}
    \frac{|x(t)|}{|y(t)|}
  &= \frac{|x(T_{2k+2})+\delta(t-T_{2k+2})y(T_{2k+2})|}{|y(T_{2k+2})|}
  \ge \frac{|x(T_{2k+2})|}{|y(T_{2k+2})|}=\\
  &=  e^{-(T_{2k+2}-T_{2k+1})} \frac{|x(T_{2k+1})|}{|y(T_{2k+1})|}\stackrel{\eqref{rel1}}{\ge} e^{-(T_{2k+2}-T_{2k+1})}\delta(T_{2k+1}-T_{2k})\stackrel{\eqref{rel2}}{\ge} \frac{\delta}{\varepsilon}\,.
  \end{align*}
  Next for $t\in [T_{2k+3},T_{2k+4}]$ and $k\ge K$, we have,
  using the inequality \eqref{rel1}, that
  \[ \frac{|x(t)|}{|y(t)|}
  = e^{-(t-T_{2k+3})}\frac{|x(T_{2k+3})|}{|y(T_{2k+3})|}
  \ge \delta(T_{2k+3}-T_{2k+2})e^{-(T_{2k+4}-T_{2k+3})}
  \ge \frac{\delta}{\varepsilon},\]
  which shows \eqref{Eq3}.

  To conclude the proof, we use the maximum norm in $\R^2$. With respect to this norm,
  $\|X(t)\xi\|=|x(t)|$, whenever $t\ge T_{2K+2}$. Our aim is to show that
  \begin{equation}\label{Eq4}
  e^{-(t-s)}
  \leq
  \frac{\|X(t)\xi\|}{\|X(s)\xi\|}
  \leq
  e^{(-1+\varepsilon)(t-s)} \fa t\geq s\geq T_{2K+2}\,.
  \end{equation}
  Equivalently, we prove \eqref{Eq4} for $k\geq K+1$ and for $t,s\in [T_{2k},T_{2k+2}]$ such that $t\geq s$. For
  $T_{2k}\le s\le t\le T_{2k+1}$, we have
  \begin{align*}
  \frac{\|X(t)\xi\|}{\|X(s)\xi\|}
  &=\frac{|x(t)|}{|x(s)|}
  =e^{-(t-s)}\frac{x(T_{2k})+\delta(t-T_{2k})y(T_{2k})}{x(T_{2k})+\delta(s-T_{2k})y(T_{2k})}\\
  &=e^{-(t-s)}\frac{1+\delta(t-T_{2k})y(T_{2k})/x(T_{2k})}{1+\delta(s-T_{2k})y(T_{2k})/x(T_{2k})}\\
  &=e^{-(t-s)}\left(1+\frac{\delta(t-s)y(T_{2k})/x(T_{2k})}
  {1+\delta(s-T_{2k})y(T_{2k})|/x(T_{2k})}\right)\,,
  \end{align*}
  which together with \eqref{Eq3} implies that
  \[
  e^{-(t-s)}\le
  \frac{\|X(t)\xi\|}{\|X(s)\xi\|}
  \leq
  e^{-(t-s)}(1+\eps(t-s))
  \leq
  e^{(-1+\varepsilon)(t-s)}\,.
  \]
  If $T_{2k+1}\le s\le t\le T_{2k+2}$, then
  \[
  \frac{\|X(t)\xi\|}{\|X(s)\xi\|}
  =
  \frac{|x(t)|}{|x(s)|} \le e^{-(t-s)}.
  \]
  This means that \eqref{Eq4} is proved.
  Consequently, $\Sigma_{\xi}\subset [-1,-1+\eps]$. Letting $\eps\to 0$ leads to $\Sigma_{\xi}=\{-1\}$ and finishes the proof of this proposition.
\end{proof}

\subsection{Coincidence of the Bohl and Sacker--Sell spectrum in special cases}

We first show that the Bohl spectrum is a subset of the Sacker--Sell spectrum.
We then show that the two spectra coincide when the Sacker--Sell spectral intervals are
singletons. Finally, we show that the Bohl and Sacker--Sell spectra coincide for bounded
diagonalizable, and hence, bounded integrally separated systems.

Let $\xi,\eta\in\R^d\setminus\{0\}$. Then the two solutions $X(t)\xi$ and $X(t)\eta$ of \eqref{Eq1} are said to be \emph{integrally separated}
if there exists $K\ge 1$ and $\alpha>0$ such that
\begin{equation}\label{Eq14}
\frac{\|X(t)\xi\|}{\|X(s)\xi\|}
\ge
Ke^{\alpha(t-s)}
\frac{\|X(t)\eta\|}{\|X(s)\eta\|}
\fa t\ge s\ge 0
\end{equation}
(see e.g.~\cite[Definition 5.3.1]{Adrianova_95_1}). If $A(t)$ is bounded, then the angle between two such
solutions is bounded below by a positive number.

In the next lemma, we show that when the solutions are integrally separated and $X(t)\xi$ is the bigger solution in the above sense, then the Bohl spectrum of any non-trivial linear combination of $X(t)\xi$ and $X(t)\eta$ is always given by $\Sigma_\xi$.

\begin{lemma}\label{Lemma4}
  Consider $\xi,\eta\in\R^d\setminus\{0\}$ such that the two solutions $X(t)\xi$
  and $X(t)\eta$ of \eqref{Eq1} are integrally separated, i.e.~the inequality \eqref{Eq14} holds. Then
  \begin{equation}\label{Eq15}
    \Sigma_{\lambda \xi+\mu \eta}= \Sigma_{\xi}\fa \lambda\in\R\setminus\{0\}\mand  \mu\in\R\,.
  \end{equation}
\end{lemma}
\begin{proof}
The lemma is clear for $\mu=0$. For the rest, we may prove \eqref{Eq15}
for the case that $\lambda=\mu=1$.  By taking $s=0$ in \eqref{Eq14}, there exists $T>0$ such that
\begin{equation*}
  \frac{\|X(t)\eta\|}{\|X(t)\xi\|}\leq \frac{1}{2}\fa t\geq T\,.
\end{equation*}
Thus, for all $t\geq T$, we have
\begin{equation}\label{Eq16}
  \frac{1}{2}\leq
  1-\frac{\|X(t)\eta\|}{\|X(t)\xi\|}
  \leq
  \frac{\|X(t)(\xi+\eta)\|}{\|X(t)\xi\|}
  \leq
  1+\frac{\|X(t)\eta\|}{\|X(t)\xi\|}\leq \frac{3}{2}\,.
\end{equation}
By \eqref{Eq16}, for all $t\geq s\geq T$, we have
\[
  \frac{\|X(t)(\xi+\eta)\|}{\|X(s)(\xi+\eta)\|}
  = \frac{\|X(t)(\xi+\eta)\|}{\|X(t)\xi\|}\frac{\|X(t)\xi\|}{\|X(s)\xi\|}\frac{\|X(s)\xi\|}{\|X(s)(\xi+\eta)\|}\le 3\frac{\|X(t)\xi\|}{\|X(s)\xi\|}\,,
\]
which implies that
\begin{equation}\label{1}
\frac{1}{t-s}\ln\frac{\|X(t)(\xi+\eta)\|}{\|X(s)(\xi+\eta)\|}\le \frac{\ln 3}{t-s}+\frac{1}{t-s}\ln\frac{\|X(t)\xi\|}{\|X(s)\xi\|}\,.
\end{equation}
Conversely, for all $t\ge s\ge T$, we have
\[
\frac{\|X(t)(\xi+\eta)\|}{\|X(s)(\xi+\eta)\|}
=
\frac{\|X(t)(\xi+\eta)\|}{\|X(t)\xi\|}\frac{\|X(t)\xi\|}{\|X(s)\xi\|}\frac{\|X(s)\xi\|}{\|X(s)(\xi+\eta)\|}\ge \frac{1}{3}\frac{\|X(t)\xi\|}{\|X(s)\xi\|}\,,
\]
so that
\begin{equation}\label{2}
 \frac{1}{t-s}\ln\frac{\|X(t)(\xi+\eta)\|}{\|X(s)(\xi+\eta)\|}\ge -\frac{\ln 3}{t-s}+\frac{1}{t-s}\ln\frac{\|X(t)\xi\|}{\|X(s)\xi\|}\,.
 \end{equation}
Let $\set{t_n}_{n\in\N}$ and $\set{s_n}_{n\in\N}$ be two positive sequences with $\lim_{n\to\infty} (t_n-s_n)=\infty$ and $\lim_{n\to\infty}s_n=\infty$. Since $\lim_{n\to\infty}t_n=\lim_{n\to\infty}s_n=\infty$, there exists $N\in\N$ such that $t_n\geq s_n\geq T$ for all $n\geq N$. Hence, combining \eqref{1} and \eqref{2} yields
\[
\lim_{n\to\infty}\frac{1}{t_n-s_n}\ln\frac{\|X(t_n)(\xi+\eta)\|}{\|X(s_n)(\xi+\eta)\|}
=
\lim_{n\to\infty}\frac{1}{t_n-s_n}\ln\frac{\|X(t_n)\xi\|}{\|X(s_n)\xi\|}\,,
\]
whenever one of the two above limits exists. This fact, together with Lemma~\ref{Lemma}~(i), shows that $\Sigma_{\xi+\eta}=\Sigma_{\xi}$.
This concludes the proof of this lemma.
\end{proof}

We first use this lemma to show that the Bohl spectrum is a subset of the Sacker--Sell spectrum. As a consequence, the filtration corresponding to the Bohl spectrum is finer than the filtration corresponding to Sacker--Sell spectrum.

\begin{theorem} \label{ken1}
  Consider the Bohl spectrum $\Sigma_{\rm Bohl}$ and the Sacker--Sell spectrum $\Sigma_{\rm SS}$ of a linear nonautonomous differential equation \eqref{Eq1}. The following statements hold:
  \begin{itemize}
  \item [(i)] The Bohl spectrum is a subset of the Sacker--Sell spectrum.
  \item [(ii)] The filtration associated with the Bohl spectrum is finer than the one of Sacker--Sell spectrum.
  \end{itemize}
\end{theorem}

\begin{proof}
(i) Let $\lambda\in\R\setminus \Sigma_{\rm SS}$ be arbitrary. Then $\dot x = A(t)x$ has an
exponential dichotomy with growth rate $\lambda$, which means that there exists a projector $P\in\R^{d\times d}$ such that
\begin{equation}\label{Eq18}
\|X(t)PX^{-1}(s)\|\le Ke^{(\lambda-\alpha)(t-s)}\fa 0\le s\le t
\end{equation}
and
\begin{equation}\label{Eq19}
\|X(t)(\mathds{1}-P)X^{-1}(s)\|
\le
Ke^{-(\lambda+\alpha)(s-t)}\fa 0\le t\le s\,.
\end{equation}
Then for any $\xi\in \ker P\setminus\{0\}$ and $\eta\in \hbox{\rm im}\, P\setminus\{0\}$, the solutions $X(t)\xi$ and $X(t)\eta$ are integrally separated. So, by virtue of Lemma \ref{Lemma4}, we have
\[
\Sigma_{\rm Bohl}
=
\bigcup_{\xi\in \ker P\setminus\{0\}}\Sigma_{\xi}
\cup
\bigcup_{\eta\in \operatorname{im}\, P\setminus\{0\}}\Sigma_{\eta}\,.
\]
From \eqref{Eq18}, we derive that for all $\eta \in \hbox{\rm im}\, P\setminus\{0\}$,
\[
\|X(t)\eta\|=\|X(t)P\eta\|\le Ke^{(\lambda-\alpha)(t-s)}\|X(s)\eta\|\,,
\]
which implies that $\Sigma_{\eta}\subset (-\infty,\lambda-\alpha]$. Similarly, using \eqref{Eq19}, we obtain that
if $\xi$ is in the kernel of $P$, then $\Sigma_{\xi}\subset [\lambda+\alpha,\infty)$. Thus, $\lambda\not\in \Sigma_{\rm Bohl}$, which finishes the proof of (i).

(ii) Let $I_j$ be the rightmost component of the Bohl spectrum contained in Sacker--Sell spectral interval $[a_i,b_i]$, and let $\mathcal S_j$ be the subspace
in the Bohl filtration corresponding to the union of $I_j$ and the intervals to its left. Then $\mathcal S_j=\cM_{\lambda}=\setb{\xi \in\R^d: \Sigma_{\xi}\subset [-\infty,\lambda)}$ for
$\lambda\in (b_i,a_{i+1})$, where $\cM_\lambda$ is defined as in \eqref{StableSpace}. If $\xi\in S_j$, then for each $\lambda\in (b_i,a_{i+1})$, there exists $K\ge 1$ such that
\begin{displaymath}
  \|X(t)\xi\|\le Ke^{\lambda(t-s)}\|X(s)\xi\| \fa t\ge s\,.
\end{displaymath}
It follows that $\xi$ is in the pseudo-stable subspace for the exponential dichotomy with growth rate $\lambda$ of $\dot x=A(t)x$ for $\lambda\in (b_i,a_{i+1})$. Hence,
$\xi\in \mathcal W_i$, which proves $\mathcal S_{j}\subset \mathcal W_{i}$, with $\mathcal W_i$ defined as in Theorem~\ref{theo3}. Conversely, note that since  $\mathcal W_{i}$ is the pseudo-stable subspace for the exponential dichotomy with growth rate $\lambda$ of $\dot x=A(t)x$ for $\lambda\in (b_i,a_{i+1})$, there exist
constants $K, \alpha>0$ such that for all  $\xi\in \mathcal W_{i}$, we have
\begin{displaymath}
  \|X(t)\xi\|\le Ke^{(\lambda-\alpha)(t-s)}\|X(s)\xi\|  \fa t\ge s\,.
\end{displaymath}
This means that $\Sigma_{\xi}\subset (-\infty,\lambda)$ for all $\lambda\in (b_i,a_{i+1})$
which implies that $\xi\in \mathcal S_{j}$. Thus $\mathcal S_{j}=\mathcal W_{i}$.
In fact, what we have proved is that the Bohl filtration is $\mathcal S_{i}$, $i\in\set{1,\dots,m}$, and the Sacker--Sell filtration is $\mathcal W_{i}$,
for $i\in\set{1,\ldots,n}$, where $n\le m\le d$, and there exist
$i_{1}<i_{2}<\cdots<i_{n}$ such that $\mathcal S_{i_{j}}=\mathcal W_{j}$.
\end{proof}
\begin{theorem}[Bohl and Sacker--Sell spectra of diagonalizable systems]\label{theo1}
  Suppose that the bounded linear nonautonomous differential equation \eqref{Eq1} is \emph{diagonalizable}, i.e.~it is kinematically similar to a (nonautonomous) diagonal system.
  Then the Bohl and Sacker--Sell spectrum of \eqref{Eq1} coincide. In particular, both spectra coincide for bounded one-dimensional systems.
\end{theorem}

\begin{proof}
  By assumption, the linear system \eqref{Eq1} is kinematically similar to a diagonal system
  \begin{equation}\label{Eq17}
    \dot x=D(t)x={\rm diag}(a_1(t),\ldots, a_d(t))\,x\,,
  \end{equation}
  where the $a_i(t)$ are bounded.  Since the Bohl and Sacker--Sell spectra are invariant under kinematic similarity,
  it is sufficient to show that the Bohl spectrum $\Sigma_{\rm Bohl}$ and the dichotomy spectrum
  $\Sigma_{\rm SS}$ of \eqref{Eq17} coincide.
  For $i\in\set{1,\dots, d}$, define
  \[
    \alpha_i:=\liminf_{t-s\to\infty}\frac{1}{t-s}\int^t_sa_{i}(u)\,\rmd u
    \qandq
    \beta_i=\limsup_{t-s\to\infty}\frac{1}{t-s}\int^t_sa_{i}(u)\,\rmd u\,.
  \]
  It follows from Proposition~\ref{prop1} that
  \[
    \Sigma_{\rm SS}=\bigcup_{i=1}^d\,[\alpha_i,\beta_i]\,.
  \]
%
  To compute $\Sigma_{\rm Bohl}$, let $(e_1,\dots,e_d)$ denote the standard orthonormal basis of $\R^d$.
  A simple computation yields that
  \[
  \Sigma_{e_i}=[\alpha_i,\beta_i]\fa i\in\set{1,\dots,d}\,,
  \]
  which implies
  \[\bigcup_{i=1}^d[\alpha_i,\beta_i]\subset \Sigma_{\rm Bohl}\subset \Sigma_{\rm SS}
  =\bigcup_{i=1}^d[\alpha_i,\beta_i]\,,\]
  and completes the proof.
\end{proof}

\begin{remark}
  Proposition~\ref{Non-coincidence} and Theorem~\ref{theo1} also show that the Bohl spectrum of a bounded upper triangular
  system is, in general, not equal to that for the diagonal part, unlike the situation for the Sacker--Sell
  spectrum in the bounded half-line case (see also Proposition~\ref{prop1}). However the Bohl spectrum of the triangular system is a subset of the
  Sacker--Sell spectrum (see Theorem~\ref{ken1} above), which equals the Sacker--Sell spectrum of the diagonal part, and
  the Sacker--Sell spectrum of the diagonal part coincides with its Bohl spectrum (see Theorem~\ref{theo1} above). We conclude
  that for bounded systems, the Bohl spectrum of an upper triangular system is a subset of the Bohl spectrum of its diagonal part.
\end{remark}

The linear nonautonomous differential equation \eqref{Eq1} is said to be \emph{integrally separated} if there are $d$ independent
solutions $X(t)\xi_1,\dots,X(t)\xi_d$ such that $X(t)\xi_i$ and $X(t)\xi_{i+1}$ are
integrally separated for all $i\in\set{1,\dots,d-1}$.

We now prove using the previous theorem that the Bohl and Sacker--Sell spectra coincide for bounded integrally separated
systems. This means also that the Bohl spectrum depends continuously on parameters for such systems.
\begin{corollary}
  Suppose that system \eqref{Eq1} is integrally separated, and $A(t)$ is bounded in $t\in\R^+_0$.  Then the Bohl spectrum coincides with the Sacker--Sell spectrum of \eqref{Eq1}.
\end{corollary}

\begin{proof}
  By Bylov's Theorem \cite[Theorem~5.3.1, p.~149]{Adrianova_95_1}, the linear system \eqref{Eq1} is
  kinematically similar to a bounded diagonal system
  \begin{displaymath}
    \dot x=D(t)x={\rm diag}(a_1(t),\ldots, a_d(t))x\,,
  \end{displaymath}
  and a direct application of Theorem~\ref{theo1} completes the proof.
\end{proof}

\begin{remark}
  The boundedness assumption of $A(t)$ in the above corollary is needed, since there exists an  unbounded integrally separated system which is not diagonalizable such that its Bohl spectrum and and its Sacker--Sell spectrum are different. Consider the system $\dot x=A(t)x$, where $A(t)$ is defined by
  \[
  A(t)
  :=
    \begin{pmatrix}
      0 &  2e^{t} \\
      0 & 1 \\
    \end{pmatrix}
  \fa t \ge 0\,.
  \]
  The fundamental matrix solution $X(t)$ of this system is given by
  \[
  X(t)=
         \begin{pmatrix}
           1 & e^{2t}-1 \\
           0 & e^t \\
         \end{pmatrix}
       \fa t \ge0\,.
  \]
  Note that
  \[
  X(t)\begin{pmatrix}
          1 \\
          0
        \end{pmatrix}
      =
        \begin{pmatrix}
  1\\
  0\\
        \end{pmatrix}
      \qandq
  X(t)
        \begin{pmatrix}
          0 \\
          1 \\
        \end{pmatrix}
      =
        \begin{pmatrix}
  e^{2t}-1\\
  e^t\\
        \end{pmatrix}
      ,
  \]
  which implies that these two solutions are integrally separated.
  It follows from Lemma~\ref{Lemma4} that $\Sigma_{{\rm Bohl}}=\{0\}\cup\{2\}$, and by explicit presentation of $X(t)$, we see that the system is not reducible and hence $\Sigma_{{\rm SS}}$
  is an interval containing the points $0$ and $2$.
\end{remark}
Let $\mathcal B$ denote the linear space of bounded measurable matrix-valued functions  $A:\R_0^+\to \R^{d\times d}$. We endow $\mathcal B$ with the $L^{\infty}$-norm defined by
\begin{displaymath}
  \|A-B\|_{\infty}=\operatorname{ess\,sup}_{t\in\R_0^+}\|A(t)-B(t)\|\,,
\end{displaymath}
so that $(\mathcal B,\|\cdot\|_{\infty})$ is a Banach space. Using \cite{Millionscikov_69_1}, one can show that there exists an open and dense set $\mathcal R$ of $\mathcal B$ such that for all $A\in\mathcal R$, the associated linear nonautonomous differential equation is integrally separated (note that genericity of exponential dichotomies for two-dimensional quasi-periodic linear systems was treated in \cite{Fabbri_00_1}). As a consequence, we obtain the following corollary.

\begin{corollary}[Coincidence is generic]\label{Cor1}
  The Bohl spectrum and the Sacker--Sell spectrum coincide generically for bounded linear nonautonomous differential equations.
\end{corollary}

We demonstrate by means of a counterexample that the Bohl spectrum is not ever upper semi-continuous in general with perturbations to the right-hand side in the $L^{\infty}$-norm. Note that the Sacker--Sell spectrum is upper semi-continuous in general, and in \cite{Poetzsche_Unpub_1}, sufficient criteria for continuity of the Sacker--Sell spectrum are established.

\begin{corollary}[Discontinuity of the Bohl spectrum]
The mapping $A\mapsto \Sigma_{\rm Bohl}(A)$ is not upper semi-continuous in general.
\end{corollary}

\begin{proof}
  Consider the linear system \eqref{Eq2}, and for $\varepsilon \in \R$, define the perturbations
  \begin{displaymath}
    A_{\varepsilon}(t):=
    \begin{pmatrix}
      -1&\delta\cr 0&-1+\varepsilon
    \end{pmatrix}
    \fa t\in [T_{2k}, T_{2k+1}]
  \end{displaymath}
  and
  \begin{displaymath}
    A_{\varepsilon}(t) := \begin{pmatrix}-1&0\cr 0&\varepsilon\end{pmatrix} \fa t\in [T_{2k+1}, T_{2k+2}]\,.
  \end{displaymath}
  Looking at the diagonal, we see that this system has the Sacker--Sell spectrum $\{-1\}\cup [-1+\varepsilon,\varepsilon]$. In particular, for $\eps>0$, it follows
  that the system is integrally separated, and hence, the Bohl spectrum is also $\{-1\}\cup [-1+\varepsilon,\varepsilon]$. However, the Bohl spectrum for $\varepsilon = 0$ is given by
  $\{-1\}$ (see Proposition~\ref{Non-coincidence}), so the Bohl spectrum is not upper semi-continuous at $\varepsilon = 0$.
\end{proof}

Suppose the Sacker--Sell spectrum consists of points. Then by Theorem~\ref{ken1}, the Bohl spectrum consists of points. We still need to prove each point in
the Sacker--Sell spectrum is also in the Bohl spectrum. This follows from the next lemma.

\begin{lemma}
  Let $[a,b]$ be a spectral interval of the Sacker--Sell spectrum of the linear nonautonomous differential equation \eqref{Eq1}. Then there exists a solution whose Bohl spectrum is contained in $[a,b]$.
\end{lemma}

\begin{proof}
  Let $\set{I_i=[a_i,b_i]}_{i\in\set{1,\dots,k}}$ be the ordered Sacker--Sell spectral intervals of \eqref{Eq1}. Consider the filtration
  \[
  \{0\}=\mathcal W_0 \subsetneq\mathcal W_1\subsetneq\mathcal W_2\subsetneq\dots\subsetneq \mathcal W_k =\R^d\,,
  \]
  established in Theorem~\ref{theo3}, satisfying the dynamical characterization
  \begin{displaymath}
    \mathcal W_i= \setB{\textstyle\xi\in \R^d: \sup_{t\in\R^+_0} \|X(t)\xi\|e^{-\gamma t} < \infty}
  \end{displaymath}
  for all $i\in\set{1,\dots,k-1}$ and $\gamma\in (b_i,a_{i+1})$. There exists an $i\in\set{1,\dots,k}$ such that $I_i = [a,b] = [a_i,b_i]$. Note that $\mathcal W_{i-1}$ is a proper subspace of $\mathcal W_i$ and we can write $\mathcal W_i=\mathcal W_{i-1}\oplus \mathcal V$
  with $\mathcal V \not= \set{0}$. Now $\dot x=A(t)x$ has an exponential dichotomy with growth rate $b+\varepsilon$ with pseudo-stable subspace $\mathcal W_i$. This means that for
  all $\xi\in \mathcal W_i$, there exist $K_1>0$ and $\alpha_1>0$ such that
  \begin{equation}\label{inequ1}
    \frac{\|X(t)\xi\|}{\|X(s)\xi\|} \le K_1e^{(b+\varepsilon-\alpha_1)(t-s)} \fa t\ge s\ge 0\,.
  \end{equation}
  Next $\dot x=A(t)x$ has an exponential dichotomy with growth rate $a-\varepsilon$ with a pseudo-unstable subspace
  $\mathcal V$ \cite[Remark~5.6 and Lemma~6.1]{Rasmussen_09_1}. This means that for all $\xi\in \mathcal V$, there exist $K_2>0$ and $\alpha_2>0$ such that
  \begin{equation}\label{inequ2}
    \frac{\|X(t)\xi\|}{\|X(s)\xi\|} \ge K_2e^{(a-\varepsilon+\alpha_2)(t-s)} \fa t\ge s\ge0\,.
  \end{equation}
  From \eqref{inequ1}, it follows that
  \[ \limsup_{t-s\to\infty}\frac{1}{t-s}\ln\frac{\|X(t)\xi\|}{\|X(s)\xi\|}
    \le b+\varepsilon-\alpha_1<b+\varepsilon\,,\]
  and from \eqref{inequ2}, it follows that
  \[ \liminf_{t-s\to\infty}\frac{1}{t-s}\ln\frac{\|X(t)\xi\|}{\|X(s)\xi\|}
    \ge a-\varepsilon+\alpha_2>a-\varepsilon\,.\]
  Since $\varepsilon>0$ was chosen arbitrarily, it follows that
  \[ a\le \liminf_{t-s\to\infty}\frac{1}{t-s}\ln \frac{\|X(t)\xi\|}{\|X(s)\xi\|}
  \le  \limsup_{t-s\to\infty}\frac{1}{t-s}\ln\frac{\|X(t)\xi\|}{\|X(s)\xi\|}
    \le b\,.\]
  Thus, $\Sigma_{\xi}\subset [a,b]$.
\end{proof}

\begin{corollary} If the Sacker--Sell spectrum consists of points, then
it coincides with the Bohl spectrum.
\end{corollary}

\begin{remark}
  Each component of the Sacker--Sell spectrum contains points of the Bohl spectrum.
  One may ask how many components of the Bohl spectrum can there be in a Sacker--Sell spectral interval.
  For a bounded integrally separated system, the answer is one since the two spectra coincide.
  For bounded systems in two dimensions, that leaves us with the case where the Sacker--Sell spectrum is one interval, and the system
  is not integrally separated. Then if the Bohl spectrum had two components,
  we would have two integrally separated solutions. So there can only be one component.
  However in three dimensions, consider the system
  \[ \dot x=a(t)x,\quad \dot y=A(t)y,\]
  where the first is a scalar system with Bohl spectrum equal to the Sacker--Sell spectrum, given by $[-\frac{1}{2},\frac{1}{2}]$
  and the second is the two-dimensional system, we constructed in Subsection~\ref{subsec1} with Sacker--Sell spectrum $=[-1,0]$ and
  Bohl spectrum $\{-1\}$. Then the three-dimensional system has Sacker--Sell spectrum $[-1,\frac{1}{2}]$, but the
  Bohl spectrum is given by $[-\frac{1}{2},\frac{1}{2}] \cup \set{-1}$, where we have used Lemma~\ref{Lemma4}.
\end{remark}

\section{Nonlinear perturbations}\label{sec6}

This section is devoted to study whether the trivial solution of a nonlinearly perturbed system with negative Bohl spectrum is asymptotically stable. Note that if the Sacker--Sell spectrum is negative, then nonlinear stability follows directly, but we will show below by means of a counter example that we cannot obtain such a result for the Bohl spectrum. Before doing so, we look at the example from Subsection~\ref{subsec1} with negative Bohl spectrum, and we prove that the system is exponentially stable for any nonlinear perturbation. Since the Sacker--Sell spectrum of this linear system is not negative, this shows that even in those cases, stability for the nonlinear system can follow. Despite the fact that negative Bohl spectrum does not imply nonlinear stability, in a forthcoming paper, we will discuss additional conditions on the nonlinearity that guarantee nonlinear stability for systems with negative Bohl spectrum, which include cases where the Sacker--Sell spectrum cannot indicate stability.

\begin{proposition}\label{prop2}
  Consider the nonlinear differential equation
  \begin{equation}\label{eqn1}
    \dot x = A(t)x+ f(t,x)\,,
  \end{equation}
  where $A:\R^+_0\to\R^{d\times d}$ is given as in \eqref{Example_02}, and $f:\R^+_0\times \R^d\to\R^d$ is continuous with
  \begin{displaymath}
    \|f(t,x)\|\le L\| x\|^q \fa t\in\R^+_0\mand x\in B_\delta(0)
  \end{displaymath}
  for some $\delta>0$, $L\ge 1$ and $q>1$. Then the trivial solution of \eqref{eqn1} is exponentially stable, i.e.~there exist $\alpha > 0$ and $\tilde\delta > 0$ such that
  \begin{displaymath}
    \|\varphi(t,0,x)\| \le K e^{-\alpha t}\|x\|  \fa t\in\R^+_0 \mand x\in B_{\tilde\delta}(0)\,,
  \end{displaymath}
  where $\varphi$ denotes the general solution of \eqref{eqn1}.
\end{proposition}

\begin{proof}
  Since the Bohl spectrum is given by $\set{-1}$, both Lyapunov exponents must be $-1$. However, the sum of the Lyapunov exponents is bounded below by
  \[ \limsup_{t\to\infty}\frac{1}{t}\int^t_0{\rm Tr}\,A(u)\,\rmd u\,,\]
  see \cite[Theorem~2.5.1]{Adrianova_95_1} and \cite[p.~226]{Barreira_08_1}. For our system, this means
  \[ -2 \ge \limsup_{t\to\infty}\frac{1}{t}\int^t_0 (-1+a_{22}(u))\,\rmd u\,,\]
  and hence that
  \[ -1\ge \limsup_{t\to\infty}\frac{1}{t}\int^t_0 a_{22}(u)\,\rmd u\,.\]
  On the other hand, since $a_{22}(t)\ge -1$ for all $t\ge0$, it follows that
  \[ -1 \le \liminf_{t\to\infty}\frac{1}{t}\int^t_0 a_{22}(u)\,\rmd u\,.\]
  We conclude that
  \[ \lim_{t\to\infty}\frac{1}{t}\int^t_0 a_{22}(u)\,\rmd u=-1\]
  and so we get regularity from \cite[Theorem~64.2]{Hahn_67_1} or \cite[Theorem~3.8.1]{Adrianova_95_1}. Then our system
  is regular and has negative Lyapunov exponents, so for any higher-order perturbation, the zero solution is exponentially
  stable (see \cite{Lyapunov_66_1}, \cite[Theorem~65.3]{Hahn_67_1} or \cite{Barreira_13_1}).
\end{proof}

\begin{remark}
  Consider the linear system \eqref{Eq2} used in the above proposition, and let $a<b$. Then the system
  \[
    \dot x=\big((b-a)A((b-a)t)+b\big)x
  \]
  has Sacker--Sell spectrum $[a,b]$ and Bohl spectrum $\{a\}$ since either spectrum
  of $\dot x=\gamma A(\gamma t)x$ is $\gamma$ times the spectrum of $\dot x=A(t)x$,
  and the spectrum of $\dot x=(A(t)+b)x$ is the translation of the spectrum of $\dot x=A(t)x$ by the number $b$. Taking $[a,b] = [-1+\eps,\eps]$, where $0<\eps<1$, implies that we get a system with $\Sigma_{\rm Bohl}=\set{-1+\eps}$ and $\Sigma_{\rm SS} = [-1+\eps,\eps]$, and we obtain asymptotic stability for nonlinear perturbations similarly to Proposition~\ref{prop2}, although the Sacker--Sell spectrum has nontrivial intersection with the position half line.
\end{remark}

We now study an example for which that Bohl spectrum is negative, and there exists a nonlinear perturbation such that the trivial solution of the nonlinear system is not asymptotically stable.
Let $\alpha,\beta,\gamma$ and $\delta$ be positive real numbers with
\begin{equation}\label{AddRemark_01}
  \beta>2\gamma+3\alpha\,,\, \gamma>2\alpha \qandq\delta\geq 1\,.
\end{equation}
Define a piecewise constant matrix-valued function $A:\R_0^+\rightarrow \R^{2\times 2}$ by
\begin{equation}\label{Example_02a}
  A(t):=\left\{
  \begin{array}{ccl}
    A_1 &  : & t\in [0,1) \hbox{ or } 2^{2k+1}\le t < 2^{2k+2}\,, \\
    A_2 & : &  2^{2k}\le t < 2^{2k+1}\,,
  \end{array}\right.
\end{equation}
where $k\in\N_0$ and
\[
  A_1:=\left(\begin{matrix} -\alpha& \delta\\ 0&-\beta\end{matrix}\right)\qandq
  A_2:=\left(\begin{matrix} -\alpha& \delta\\ 0&\gamma\end{matrix}\right)\,.
\]
We now compute the Bohl spectrum of the system
\begin{equation}\label{AddRemark_Eq02}
  \dot x=A(t)x\,,
\end{equation}
where $A:\R_0^+\to\R^{2\times 2}$ is defined as in \eqref{Example_02a}. We need the following preparatory result.

\begin{lemma}\label{Add_Lemma}
  Let $(x(t),y(t))$ be a solution of \eqref{AddRemark_Eq02}. Then
  \[
  |y(t)|
  \leq
  e^{-\frac{\beta-2\gamma}{3}t}|y(0)|
  \fa t\geq 0\,.
  \]
\end{lemma}

\begin{proof}
For any $k\in\N_0$, we have
\[
y(2^{2k+2})= e^{-\beta  2^{2k+1}} y(2^{2k+1})
\qandq
y(2^{2k+1})= e^{\gamma 2^{2k} } y(2^{2k})\,,
\]
which implies that
\[
y(2^{2k+2})
=
e^{-(2\beta-\gamma)2^{2k}}y(2^{2k}).
\]
Hence,
\begin{align*}
y(2^{2k})
&=
e^{-(2\beta-\gamma)(2^{2k-2}+\dots+ 2^0)}y(1)\\
&=
e^{-\frac{2\beta-\gamma}{3}2^{2k}}e^{-\frac{\beta+\gamma}{3}} y(0)\,,
\end{align*}
and
\[
y(2^{2k+1})
=
e^{2^{2k}\gamma}y(2^{2k})
=
e^{-\frac{\beta-2\gamma}{3}2^{2k+1}}e^{-\frac{\beta+\gamma}{3}} y(0).
\]
Consequently, if $2^{2k}\leq t < 2^{2k+1}$, then
\begin{align*}
|y(t)|
&=
e^{\gamma(t-2^{2k})} |y(2^{2k})|\\
&=
e^{\gamma(t-2^{2k})}  e^{-\frac{2\beta-\gamma}{3}2^{2k}}e^{-\frac{\beta+\gamma}{3}} |y(0)|\\
&\leq
e^{-\frac{\beta-2\gamma}{3}t}|y(0)|,
\end{align*}
and if $2^{2k+1}\leq t < 2^{2k+2}$, then
\begin{align*}
|y(t)|
&=
e^{-\beta (t-2^{2k+1})} |y(2^{2k+1})|\\
&=
e^{-\beta (t-2^{2k+1})}  e^{-\frac{\beta-2\gamma}{3}2^{2k+1}}e^{-\frac{\beta+\gamma}{3}} |y(0)|.\\
&\leq
e^{-\frac{\beta-2\gamma}{3}t}|y(0)|,
\end{align*}
which completes the proof of this lemma.
\end{proof}

\begin{proposition}\label{Add_Proposition}
  The Bohl spectrum of \eqref{AddRemark_Eq02} satisfies $\Sigma_{\rm Bohl}\leq -\alpha<0$.
\end{proposition}
\begin{proof}
Fix an initial condition $(x_0,y_0)\in\R^2$, and let $\xi(t)=(x(t),y(t))^{\rT}$ denote the solution of \eqref{AddRemark_Eq02} with $\xi(0)=(x_0,y_0)^{\rT}$. Obviously, $\Sigma_{\xi}=\Sigma_{-\xi}$ and we thus may assume that $y_0\geq 0$. Let $\R^2$ be endowed with the maximum norm for the remainder of this proof. We consider the following two cases.

\noindent \emph{Case 1.} $y_0=0$. Then we have $\xi(t)=(e^{-\alpha t},0)^{\rT}$, which implies $\Sigma_{\xi}=\{-\alpha\}$.

\noindent \emph{Case 2.} $y_0\not=0$. By the variation of constants formula, we have
\[
x(t)=e^{-\alpha t}\left(x_0+\delta \int_0^t e^{\alpha s} y(s)\,\rmd s\right)\,,
\]
and by \eqref{AddRemark_01} and  Lemma~\ref{Add_Lemma}, the integral $\int_0^{\infty}e^{\alpha s} y(s)\,\rmd s$ exists. We divide the remainder of Case~2 into two different cases.

\noindent \emph{Case 2.1.} $x_0\not=-\delta \int_0^{\infty}e^{\alpha s} y(s)\,\rmd s$. Then
\[
  \lim_{t\to\infty} e^{\alpha t} |x(t)|=
  \left|
  x_0+\delta\int_0^{\infty}
  e^{\alpha s} y(s)\,\rmd s\right|\,.
\]
Hence, there exists a $T>0$ such that for all $t\geq T$, we have $|x(t)|>|y(t)|$. This implies
\[
\lim_{t-s\to\infty}\frac{1}{t-s}\ln\frac{\|\xi(t)\|}{\|\xi(s)\|}=-\alpha\,,
\]
which leads to $\Sigma_{\xi}=\{-\alpha\}$.

\noindent \emph{Case 2.2.} $x_0=-\delta \int_0^{\infty}e^{\alpha s} y(s)\,\rmd s$. Then
\begin{equation}\label{New_Eq3}
  x(t)=-\delta e^{-\alpha t}\int_{t}^{\infty}e^{\alpha s} y(s)\,\rmd s<0\,.
\end{equation}
Since $y(t)>0$ for all $t\geq 0$, it follows that for $t\geq s$
\begin{equation}\label{New_Eq6}
e^{\alpha t} |x(t)| =\delta\int_t^{\infty} e^{\alpha u} y(u)\,\rmd u\leq \delta \int_s^{\infty} e^{\alpha u} y(u)\,\rmd u= e^{\alpha s} |x(s)|\,.
\end{equation}
Having done this for the $x$-component of $\xi$, we now compare $e^{\alpha t}\|\xi(t)\|$ and $e^{\alpha s}\|\xi(s)\|$ with $t\ge s$. The following statements hold.
\begin{itemize}
\item[(i)] For all $t,s\in [2^{2k+1}, 2^{2k+2})$ with $t\geq s$, we have $e^{\alpha t}|y(t)|\leq e^{\alpha s} |y(s)|$, and with \eqref{New_Eq6}, we get
\[
\|\xi(t)\|\leq e^{-\alpha (t-s)}\|\xi(s)\|\,.
\]
\item[(ii)] For all $t,s\in [2^{2k}, 2^{2k+1}-1)$ with $t\geq s$ and $k\in\N$, note that in the interval $[2^{2k}, 2^{2k+1}]$ the function $y(t)$ is increasing, so we have
\[
e^{\alpha t}|x(t)|=\delta\int_t^{\infty} e^{\alpha s} y(s)\,\rmd s
\geq
\int_t^{2^{2k+1}} e^{\alpha s} y(s)\,\rmd s
\geq
e^{\alpha t} y(t)
\]
for all $t,s\in [2^{2k}, 2^{2k+1}-1)$ with $t\geq s$. Hence, $\|\xi(t)\|=|x(t)|$, and from~\eqref{New_Eq6}, we obtain
\[
\|\xi(t)\|\leq e^{-\alpha (t-s)}\|\xi(s)\|\,.
\]
\item[(iii)] For $2^{2k+1}-1\leq s \leq t\leq 2^{2k+1}$, we have
\[
\|\xi(t)\|\leq  e^{M(t-s)}\|\xi(s)\|\,,
\]
where $M=\max\{\alpha+\delta,\gamma\}$ is the operator norm of $A_2$ with respect to the maximum norm. In particular,
\begin{equation}\label{Add_Eq1}
\frac{\|\xi(2^{2k+1})\|}{\|\xi(2^{2k})\|}
=
\frac{\|\xi(2^{2k+1})\|}{\|\xi(2^{2k+1}-1)\|}
\frac{\|\xi(2^{2k+1}-1)\|}{\|\xi(2^{2k})\|}
\leq e^{M+\alpha} e^{-\alpha(2^{2k+1}-2^{2k})}\,.
\end{equation}
\end{itemize}
Let $t> s\geq 4$. We choose $m,n\in\N$ such that
\[
2^{m+1}>t\geq 2^m\geq 2^n\geq s > 2^{n-1}\,.
\]
Note that
\begin{displaymath}
\frac{\|\xi(t)\|}{\|\xi(s)\|}
=
\frac{\|\xi(t)\|}{\|\xi(2^m)\|}
\frac{\|\xi(2^n)\|}{\|\xi(s)\|}
\prod_{k=n}^{m-1}
\frac{\|\xi(2^{k+1})\|}{\|\xi(2^k)\|}\,.
\end{displaymath}
Using \eqref{Add_Eq1}, we obtain
\begin{align*}
\prod_{k=n}^{m-1}
\frac{\|\xi(2^{k+1})\|}{\|\xi(2^k)\|}
&\leq
\prod_{k=n}^{m-1}
e^{M+\alpha} e^{-\alpha(2^{2k+1}-2^{2k})}\\
&=
e^{(m-n)(M+\alpha)} e^{-\alpha (2^m-2^n)}\,.
\end{align*}
Analogously to \eqref{Add_Eq1}, we have
\[
\frac{\|\xi(t)\|}{\|\xi(2^m)\|}
\leq e^{M+\alpha} e^{-\alpha(t-2^m)}
\qandq
\frac{\|\xi(2^n)\|}{\|\xi(s)\|}
\leq e^{M+\alpha} e^{-\alpha(2^n-s)}\,.
\]
Consequently, we obtain the estimate
\[
\frac{\|\xi(t)\|}{\|\xi(s)\|}
\leq
e^{(m-n+2)(M+\alpha)}e^{-\alpha(t-s)}\,.
\]
Thus,
\begin{align*}
\limsup_{t-s\to\infty}\frac{1}{t-s}\ln\frac{\|\xi(t)\|}{\|\xi(s)\|}
&\leq
-\alpha+(M+\alpha)\limsup_{t-s\to\infty}\frac{\log_2(t/s)}{t-s}\\
&\leq
-\alpha+(M+\alpha)\limsup_{t-s\to\infty}\frac{\log_2(1+(t-s)/4)}{t-s}= -\alpha\,,
\end{align*}
which completes the proof.
\end{proof}

The following proposition shows that, although the Bohl spectrum is bounded above by $-\alpha<0$, for certain nonlinear perturbation of \eqref{AddRemark_Eq02}, the system is unstable.

\begin{proposition}\label{Theorem1}
Consider the perturbed system
\begin{equation}\label{New_Eq4}
\begin{pmatrix}
\dot x\\
\dot y
\end{pmatrix}
=
A(t)
\begin{pmatrix}
x\\
y
\end{pmatrix}
+
\begin{pmatrix}
0\\
x^2
\end{pmatrix}
\,,
\end{equation}
where $A(t)$ is defined as in \eqref{Example_02a}. Then the trivial solution of \eqref{New_Eq4} is unstable.
\end{proposition}
\begin{proof}
Let $(x_0,y_0)$ be an initial condition at time $t=0$ for the solution $(x(t),y(t))$ with $x_0>0$ and $y_0>0$. We prove $\limsup_{t\to\infty}y(t)=\infty$ with the following two steps.

\noindent \emph{Step 1}. We show that both $x(t)$ and $y(t)$ are positive for all $t\geq 0$. If we suppose the contrary, then we can define
\[
T^*:=
\inf
\big\{
t \ge 0: x(t)<0 \hbox{ or } y(t)< 0
\big\}
\]
By continuity, we derive that $T^*>0$ and $x(t), y(t)>0$ for all $t\in [0,T^*)$. We now consider two cases.

\noindent \emph{Case 1}. If $x(T^*)=0$, then by variation of constants formula we have
\[
x(T^*)=
e^{-\alpha T^*}x_0+
\delta\int_{0}^{T^*} e^{-\alpha (T^*-s)} y(s)\,\rmd s>0\,,
\]
which leads to a contradiction.

\noindent \emph{Case 2}. If $y(T^*)=0$, then $\dot y(T^*)= x(T^*)^2>0$. Thus, there exists $\eps>0$ such that $y(T^*-\eps)<0$. This contradicts the definition of $T^*$.

\noindent \emph{Step 2}. We estimate $y(t)$. From the variation of constants formula for the first component, we have
\[
x(t)\geq e^{-\alpha t} x_0\qquad\hbox{for all } t \ge 0\,.
\]
By variation of constants formula of the second component, we have
\[
y(t)=\Lambda_2(t,0) y_0+\int_{0}^t \Lambda_2(t,s)x(s)^2\,\rmd s\,,
\]
where $\Lambda_2(t,s)$ denote the transition operator of the linear system
\[
\dot y=\left\{
  \begin{array}{ccl}
    -\beta y(t) &  : & t\in [0,1) \hbox{ or } 2^{2k+1}\le t\le 2^{2k+2}\,, \\
    \gamma y(t) & : &  2^{2k}\le t\le 2^{2k+1}\,.
  \end{array}\right.
\]
Hence,
\begin{align*}
y(2^{2k+1})
& \geq
x_0^2\int_{2^{2k}}^{2^{2k+1}}
e^{\gamma (2^{2k+1}-s)} e^{-2\alpha s}\,\rmd s\\
&\geq
e^{\gamma 2^{2k+1}}\frac{e^{-(2\alpha+\gamma) 2^{2k}}- e^{-(2\alpha+\gamma) 2^{2k+1}}}{2\alpha+\gamma}x_0^2\\
&\geq
\frac{e^{(\gamma-2\alpha)2^{2k}}-1}{2\alpha+\gamma}x_0^2\,,
\end{align*}
which proves that $\limsup_{t\to\infty}y(t)=\infty$ and finishes the proof of this proposition.
\end{proof}

Note that in a forthcoming paper, we will discuss additional conditions on the nonlinearity that guarantee nonlinear stability for systems with negative Bohl spectrum, which include cases where the Sacker--Sell spectrum cannot indicate stability.



%
%


\bigskip

\noindent \textbf{Acknowledgements.} The authors are grateful to an anonymous referee for useful comments that led to an improvement of this paper.

\newcommand{\etalchar}[1]{$^{#1}$}
\providecommand{\bysame}{\leavevmode\hbox to3em{\hrulefill}\thinspace}
\providecommand{\MR}{\relax\ifhmode\unskip\space\fi MR }
\providecommand{\MRhref}[2]{%
  \href{http://www.ams.org/mathscinet-getitem?mr=#1}{#2}
}
\providecommand{\href}[2]{#2}

\end{document}